\documentclass[12pt, a4]{article}

\usepackage[applemac]{inputenc}
\usepackage[T1]{fontenc}

\usepackage{appendix}

\usepackage{amssymb}
\usepackage{graphicx}
\usepackage{epsfig}
\usepackage{here}
\usepackage{psfrag}
\usepackage{latexsym}
\usepackage{epic}
\usepackage{eepic}
\usepackage{amsmath}
\usepackage{amsthm}
\usepackage{enumerate}
\usepackage[T1]{fontenc}
\usepackage{pstcol}
\usepackage{url}

\newcommand{\bea}{\begin{eqnarray}}
\newcommand{\eea}{\end{eqnarray}}
\newcommand{\be}{\begin{equation}}
\newcommand{\ee}{\end{equation}}
\newcommand{\beann}{\begin{eqnarray*}}
\newcommand{\eeann}{\end{eqnarray*}}
\newcommand{\beal}{\begin{align}}
\newcommand{\eal}{\end{align}}

\newcommand{\nn}{\nonumber}

\newcommand{\x}{{\bf x}}
\newcommand{\y}{{\bf y}}

\def\R{{\mathbb R}}
\def\ba{\begin{array}}
\def\ea{\end{array}}
\def\bd{\begin{displaymath}}
\def\ed{\end{displaymath}}
\def\Y{{\mathbb Y}}
\def\P{{\mathbb P}}

\def\1{{\mathbf 1}}

\def\E{{\mathbb E}}
\def\F{{\mathbb F}}

\def\N{{\mathbb N}}

\def\Z{{\mathbb Z}}

\def\m{{\mathbf m}}
\def\L{{\mathcal L}}
\def\B{{\mathcal B}}

\def\fomega{{\boldsymbol \omega}}

\DeclareMathOperator{\Cov}{Cov}
\DeclareMathOperator{\Var}{Var}

\newcommand{\argmax}{\operatornamewithlimits{arg\;max}}

\DeclareMathOperator{\e}{e}

\newtheorem{thm}{Theorem}[section]
\newtheorem{cor}{Corollary}[section]
\newtheorem{lem}{Lemma}[section]
\newtheorem{prop}{Proposition}[section]
\newtheorem{rem}{Remark}[section]

\theoremstyle{definition}

\usepackage{amsmath}
\usepackage{graphicx}
\usepackage[T1]{fontenc}
\usepackage{pstcol}

\usepackage[vcentering,dvips]{geometry}

\numberwithin{equation}{section}

\begin{document}

\thispagestyle{empty}

\begin{center}
{\bf\large Likelihood Inference for a Functional Marked Point Process with Cox-Ingersoll-Ross Process Marks}
\end{center}

\vspace{4mm} \sloppy
\begin{center}
{\sc\Large Ottmar Cronie}\footnote{E-mail address to the author: ottmar@alumni.chalmers.se}\footnote{The majority of this research was conducted at Chalmers University of Technology, Sweden}\\[8pt]
{\it Anton de Kom University of Suriname\\
POB 9212, 
Leysweg 86, 
Paramaribo, 
Suriname}\\[10pt]
\end{center}
\vspace{3mm}

\begin{abstract}

This paper considers maximum likelihood inference for a functional marked point process -- the stochastic growth-interaction process -- which is an extension of the spatio-temporal growth-interaction process to the stochastic mark setting. 
As a pilot study we here consider a particular version of this extended process, which has a homogenous Poisson process as unmarked point process and shifted independent Cox-Ingersoll-Ross processes as functional marks. 
These marks have supports determined by the lifetimes generated by an immigration-death process. 
By considering a (temporally) discrete sample scheme for the marks and by considering the process' alternative evolutionary representation as a multivariate diffusion (Markovian) with jumps, the likelihood function is expressed as a product of the process' closed form transition densities. 
Additionally, under the assumption that the mark processes are started in their common stationary distribution, and under some restrictions on the underlying parameters, 
consistency and asymptotic normality of the maximum likelihood (ML) estimators 
are proved.
The ML-estimators derived from the stationarity assumption are then compared numerically to the ML-estimators derived under non-stationarity, in order to investigate the robustness of the stationarity assumption. 
To illustrate the model's use in forestry, it is fitted to a data set of Scots pines.


\end{abstract}

\newpage
\setcounter{page}{1} \setcounter{equation}{0}

\section{Introduction}\label{SectionIntroduction}

Renshaw and S\"arkk\"a presented their spatio-temporal growth-interaction (GI) process in \cite{RS1}, and usually this process is described as a spatio-temporal point process (see e.g.\ \cite{VerreJones}) with growing and interacting marks.
The GI-process then has been further studied in a series of papers (e.g.\ \cite{Comas,Cronie,CronieYu,RenshawComas,RenshawComasMateu,RS2}), where, among other things, different inference tools have been developed.

Consider some suitable spatial study region (usually some subset of $\R^2$ or a torus). 
The basis of the GI-process is spatio-temporal point process which here will be referred to as a spatial immigration-death (SID) process. 
Specifically, the SID-process lets new individuals (points) arrive to a population (the study region) according to the jumps of a Poisson process on $\R_+$, assigns iid locations to them, which are uniformly distributed in the study region, and removes the points after iid exponential times (the temporal dynamics of the SID-process constitute a so-called immigration-death process; see e.g.\ \cite{CronieYu}). 
Moreover, once an individual has received its location, centred on its location we place a closed disk/ball with some given radius (the initial mark).
As time evolves, we let the radius of each disk grow according to a given deterministic growth structure, which depends on the radius itself as well as the locations and the radii of the other (neighbouring) marked points. 
More precisely, the structure of the simultaneous growth of the radii is given by a system of ordinary differential equations (ODEs). 
Furthermore, this system of ODEs is such that when an individual has not yet arrived to the population or if it has been removed (which is governed by the SID-process), its corresponding component of the system of ODEs is set to zero. 
Considering its possible application areas, 
one important area is the dynamical modelling of a forest stand. 
Here, as time passes, new trees arrive, grow and compete with each other until they die.
When we let each growth equation (ODE component) contain an inhibitive part, 
such that the growth of a point's radius/mark is reduced when the point is surrounded (spatially) by points with large marks, 
this inhibition reflects the natural competition for nutrients and light among trees in a forest.

The description of the GI-process as a spatio-temporal marked point process is rather vague and questions regarding an appropriate representation quickly emerge. 
Following \cite{Mateu}, we will call any marked point process where the marks are function-valued a \emph{functional marked point process}, and it has been pointed out by \cite{Mateu} that the GI-process in fact should be represented as a functional marked point process. 
More specifically, it can be treated as a marked point process (see e.g.\ \cite{DaleyVereJones,Diggle,SchneiderWeil}) for which the location space is that of a spatial Poisson process and the mark space (representing the radii of the balls/disks) is given by the space of c\`adl\`ag (right continuous with existing left limits) functions (see e.g.\ \cite{Billingsley}). 
This can be realised by letting the unmarked part of this functional marked point process be given by the collection of points scattered in the study region by the underlying SID-process during the time interval we are studying the GI-process. 
Moreover, 
we see that each mark consists of three parts: 
1) the size of the mark before the individual has arrived (zero), 
2) the size when the individual is present (governed by the ODEs), and 
3) the size once the individual is removed (zero). 
Since there clearly are jumps between parts 1) and 2), and parts 2) and 3), we see that it makes sense to consider a mark space which is of c\`adl\`ag-type. 
Note that the connection between these two representations of the GI-process can be compared to the representations of a one-dimensional Poisson process as either a point process (random measure) on $R_+$ or as a Levy-process (evolutionary process). 

Naturally the growth structure of the GI-process includes some parameters that need to be estimated when the model is fitted to data. 
For the case where the process is sampled at discrete times (hence creating a time series of marked point patterns), \cite{RS2} suggested a least-squares scheme to estimate the parameters related to the growth and interaction of the marks (the parameters in the ODEs). 
This estimation was further considered in \cite{Cronie}, where a spatio-temporal edge correction was added to the estimation procedure.
Regarding the parameters of the underlying (spatial) immigration-death process, which are the arrival and death rates of individuals, these have been estimated separately by means of different maximum likelihood (ML) estimation approaches (see e.g.\ \cite{Cronie,CronieYu,GibsonRenshaw,RS2}). 
In \cite{CronieYu} the full ML-estimation of the discretely sampled immigration-death process was treated, and, besides treating some practical aspects of the estimation, 
consistency and asymptotic normality of the ML-estimators were proved.

We note that it is unlikely that all marks in a marked point pattern (e.g.\ trees in a forest stand) have the same (deterministic) underlying growth pattern, as is the assumption in the original GI-process.
For instance, when we model a forest stand, in order to reflect phenomena such as (cumulative) measurement errors and individual growth features of each tree, there should be some noise or randomness present in the part of the model which handles the growth of the marks.
Trying to rectify this lack of individuality in the mark growth structures, our main aim here will be to take a first step in the process of adding randomness to the growth of the marks.
The approach chosen here is to add a (scaled continuous time) white noise to each component of the system of mark growth ODEs in the GI-process (see ''part 2)'' above).
This will generate a set of (possibly dependent) Brownian motion driven stochastic differential equations (SDEs) with jumps at the birth and death times  (see e.g.\ \cite{KaratzasShreve,Klebaner,Protter,Shreve}). 
In other words, we will be considering a functional marked point process with marks given by diffusion processes with jumps. 
By considering the temporal evolution of such a process, we may also say that we have defined a 
multivariate stochastic jump (diffusion) process which is time shifted and parametrized by an SID-process. 
We note that turning the marks into diffusions has some clear advantages.
To begin with, as a consequence of the randomness in the mark growth equations, we may be able to write down a likelihood structure, 
which is based on sampling the marks over time, 
so that we can treat the simultaneous ML-estimation of the whole GI-process 
instead of using the separate spatial and temporal estimators previously considered. 
In particular, we may exploit the Markovianity of the diffusions and the SID-process in the derivation of a full likelihood function.
Additionally, given a temporally discrete sampling scheme,
standard tools from likelihood theory can help us derive results about the asymptotic behaviour of the estimators (see e.g.\ \cite{Hoadley,Iacus,Rao,VanDerVaart}), when the process is sampled at discrete times. 
Ultimately it may also be possible to use tools from stochastic process theory and stochastic calculus to derive other theoretical results, such as asymptotic distributional properties of functionals of the process, which in the non-stochastic version of the GI-process only have been achievable through simulations.

Here, as an initial extension of the GI-process to the setting where the marks are stochastic, we will assume that all diffusions (mark radius processes) will be independent and of the same type (thereby generated from the same set of parameters).
Note that we hereby remove the interaction between the marks which was present in the GI-process. 
Hence, we obtain a functional marked point process with independent diffusion process marks. 
We have chosen to study only the Cox-Ingersoll-Ross (CIR) process (see e.g.\ \cite{CIR,Feller,Iacus}) to describe the marks (growth of the radii of the disks).
However, any other strictly positive diffusion which meets the requirements of the modelling setting in question could be chosen.
Note that a nice property of the CIR-process (besides being strictly positive under certain restrictions) is that it is one of the few diffusions which possesses known closed form expressions for its transition densities, and since the transition densities in turn are the building blocks of the likelihood function we can obtain a closed form expression for the likelihood function.

The paper is organised as follows.
In Section \ref{SectionSGI} the stochastic GI-process is defined and in Section \ref{SectionDistribution} we move on to further discuss some of its distributional properties and its building blocks.
Then, in Section \ref{SectionLikelihood}, with the finite dimensional distributions at hand, we define the ML-estimation regime which, in Section \ref{SectionAsymptotics}, in turn allows us to look at large sample properties of ML-estimators.
Since the consistency and the asymptotic normality of the ML-estimators are proved when the mark processes are stationary, in Section \ref{SectionEvaluation} we wish to see how robust these estimators are to the stationarity assumption.
In Section \ref{SectionPines} we evaluate the estimators on the same set of Scots pine data considered in \cite{Cronie} and 
in Section \ref{SectionDiscussion} further comments/possible extensions are given as well as a general discussion of the paper.
Finally, in the Appendix proofs of the main results can be found.

\section{The SG(I) process}\label{SectionSGI}
The \emph{stochastic growth-interaction} (SGI) process $\Phi_M$ is a functional marked point process which can be considered to be a stochastic extension of the (non-stochastic) growth-interaction (GI) process (see e.g.\ \cite{Cronie,RS1,RS2,Comas}). 
Heuristically $\Phi_M$ is described in the following way.
As time evolves, balls/disks (marked points) appear in a spatial study region $W$ at stochastic times and the radii of these balls/disks change size randomly over time until they disappear after stochastic times. 
Due to the usual biological context it will be natural to refer to each point as an \emph{individual}.
As previously noted, we here consider a simplified version of the SGI-process since we do not include any interaction between the marked points, and we therefore will refer to $\Phi_M$ as the \emph{stochastic growth} (SG) process.
It should be pointed out that due to the GI-process' forestry application it is usual that we illustrate the process in such a way that its marks describe the aforementioned growth of disks in $\R^2$ (the space occupied by the tree stocks), however, this need not be the case.
For instance, when modelling the spatio-temporal development of a forest stand, one could instead consider the case where the marks are used to describe, say, the development of the height of the trees.

As was mentioned in Section \ref{SectionIntroduction}, there are essentially two ways to construct the SG-process $\Phi_M$. 
In the first representation of the SG-process, which is given in Section \ref{SectionFMPP}, $\Phi_M$ is obtained by treating it as a functional marked point process with c\`adl\`ag function-valued marks. 
The second representation (Section \ref{SectionEvolutionary}) is obtained by considering the temporal evolution of the mark processes (which are parametrized by the other relevant information). 
Note that the latter representation is of significance since it will be exploited when we develop the statistical inference for the SG-process when the marks are sampled at discrete times.

Throughout we will assume that the spatial study region is given by a subset $W\subseteq\R^d$ of $d$-dimensional Euclidean space, $d\geq1$, with Borel sets $\B(W)\subseteq\B(R^d)$ and Lebesgue measure $\nu(W)<\infty$ (note that this also includes the case of identifying the edges of a rectangle in order to construct a torus). 
Moreover, we write $\N=\{0,1,\ldots\}$ and we denote the Euclidean norm and metric by $|x|$ and $d(x,y)$, respectively, for $x,y\in\R^d$, $d\geq1$. 
Furthermore, for a given set $A$, $\1_{A}(a)=\1\{a\in A\}$ will denote the related indicator function and $|A|$ will denote the related cardinality (it will be clear from context whether we consider the norm or the cardinality). 

Following \cite{Mateu} and the construction of a functional marked point process given therein, the mark space will be given by the set $\F=D_{[0,T]}(\R)$, $T\in\R_+=(0,\infty)$, of c\`adl\`ag (right continuous with existing left limits) functions $f:[0,T]\rightarrow\R$ (or $\F=D_{[0,\infty)}(\R)$ when $f:[0,\infty)\rightarrow\R$). 
The underlying probability space will be denoted by $(\mathcal{X},\mathcal{F},\P)$.

\subsection{Functional marked point process representation of the SG-process}\label{SectionFMPP}

Assume now that the SG-process under consideration is given by a c\`adl\`ag functional marked point process $\Phi_M = \{[X_i,M_i]:X_i\in\Phi'\}$, with locations $X_i\in W$ and functional marks $M_i\in\F$. 

More specifically, we let the unmarked process $\Phi'=\{X_1,\ldots,X_N\}$ be given by a homogeneous Poisson process on $W$, with intensity $\alpha\nu(W)T$, $\alpha\in\Theta_{\alpha}\subseteq\R_+$. We note that we hereby have $N\sim Poi(\alpha\nu(W)T)$ locations $X_1,\ldots,X_N\in W$ which are iid $Uni(W)$-distributed (their indices are assigned to them according to their ''birth times'' defined below). 

We now turn to the construction of the $\F$-valued random functional marks $M_1,\ldots,M_N$, which we will require to be almost surely (a.s.)\ positive. 
In order to generate the supports $\mathrm{supp}(M_i)=\{t\in[0,T]:M_i(t)\neq0\} = [B_i,D_i)$, $i=1,\ldots,N$, conditionally on $\Phi'$ (or simply $N$), let $B_1,\ldots,B_N$ be iid $Uni(0,T)$-distributed random variables (relabelled according to ascending size) and 
let additionally $L_1,\ldots,L_N$ be iid $Exp(\mu)$-distributed, $\mu\in\Theta_{\mu}\subseteq\R_+$. 
By now defining $D_i = (B_i+L_i)\wedge T = \min\{B_i+L_i,T\}$, $i=1,\ldots,N$,
we have that $M_i(t)=0$ for all $t\notin[B_i,D_i)$ and $M_i(t)>0$ for all $t\in[B_i,D_i)$ a.s.. 

We note here that the ''birth/arrival times'' $B_1<\ldots<B_N$ form a Poisson process on $[0,T]$ with intensity $\alpha\nu(W)$. In order to use a terminology which illustrates how $\Phi_M$ can be used to model the dynamics of a population (e.g.\ a forest stand), in connection to the birth times, we additionally call $L_1,\ldots,L_N$ the ''lifetimes'' and $D_1,\ldots,D_N$ the ''death times'' of the individuals.

As previously mentioned, each mark $M_i=\{M_i(t)\}_{t\in[0,T]}$, $i=1,\ldots,N$, can be illustrated by the space which it occupies in $\R^d$ at a given time $t$. This is done by means of the ball $B_{X_i}[M_i(t)]=\{\y\in\R^d:d(X_i,\y)\leq M_i(t)\}$, with centre $X_i$ and radius $M_i(t)$. 
Hereby, for a fixed time $t\in[0,T]$, $\Phi_M$ can be illustrated as a forest stand, or rather a Boolean model (see e.g.\ \cite{StoyanKendallMecke}), by considering the union of the disks or trees $\bigcup_{i=1}^{N} B_{X_i}[M_i(t)]$ or the union of all disks $B_{X_i}[M_i(t)]$, $i=1,\ldots,N$, such that $t\in[B_i,D_i)$ (whereby we only observe ''alive individuals'').

Considering now to the actual structure put on each $M_i=\{M_i(t)\}_{t\in[0,T]}$ when $t\in[B_i,D_i)$ (i.e.\ when the $i$th individual is alive), we will let $M_i(B_i) = M_{i}^{0}$ be the initial size of the $i$th mark process and to illustrate how the construction of $\Phi_M$ originates from the GI-process, we first recall (see e.g.\ \cite{RS2,Cronie,Comas}) that in the GI-process the marks $M_i(t)$, $i=1,\ldots,N$, were set to develop deterministically according to
\bea
\label{GI}
M_i(t)&=&M_i^0 + \int_{B_i}^{D_i}dM_i(s) \\
&=& M_i^0 + \int_{B_i}^{D_i} f(M_i(s);\theta) + \sum_{i\neq j}h(M_i(s),M_j(s),X_i,X_j;\theta)\;ds,\nn
\eea
for $t\in[B_i,D_i)$. Here $f(\cdot)$ controls the growth of radius $i$ in absence of spatial competition (so-called open growth in forestry terminology), $h(\cdot)$ controls the spatial interaction between individual $i$ and the other individuals and $\theta$ is a vector of parameters which controls $f(\cdot)$ and $h(\cdot)$ (see e.g.\ \cite{Cronie,RS2}). 

Here, however, in order to initialise the more realistic growth scenario where the marks have random growth patterns, as previously mentioned, we assume that each radius grows stochastically according to an a.s.\ positive 
stochastic process (note that an $\F$-valued random variable/element is a stochastic process with c\`adl\`ag sample paths). 
By calling $t\in[0,T]$ our \emph{global time} and $t-B_i$ our $i$th \emph{local time}, 
in order to properly express $M_i(t)$ through the global time scale, we will let the processes $Y_i(t)$, $i=1,\ldots,N$, where 
\bea
\label{Diffusion}
M_i(t) =
\left\{
\begin{array}{ll}
Y_i(t-B_i)		& \text{for } t\in[B_i,D_i)\\
0			& \text{for } t\notin[B_i,D_i),
\end{array}
\right.
\eea
be given by a system of independent (time-shifted) CIR-processes (see e.g.\ \cite{CIR,Feller,Iacus}) 
\bea
\label{SDE}
dY_i(t)	&=&	\lambda\left(1 - Y_{i}(t)/K\right)dt + \sigma\sqrt{Y_{i}(t)}dW_i(t),\\
\label{CIR}
Y_i(t)	&=&	M_{i}^{0} + \int_{0}^{t} \lambda\left(1 - \frac{Y_{i}(s)}{K}\right)ds + \int_{0}^{t}\sigma\sqrt{Y_{i}(s)}dW_i(s),
\eea
with $Y_i(0)=M_{i}^{0}$, $i=1,\ldots,N$, where the $W_i(t)$'s are independent standard Brownian motions. 
We note that this is equivalent to setting  $f(x)=\lambda(1-x/K)$, $h(\cdot)=0$ and adding a stochastic integral to expression (\ref{GI}).

The parameters $(\lambda,K,\sigma)\in\Theta_{\lambda}\times\Theta_{K}\times\Theta_{\sigma}\subseteq\R_+^{3}$ in expressions (\ref{SDE}) and (\ref{CIR}) control different aspects of the growth of a radius $M_i(t)$ of a ball/disk $B_{X_i}[M_i(t)]$:
The \emph{diffusion coefficient} $\sigma$ controls the magnitude of the random individual fluctuations of the radii.
The interpretation of the remaining two parameters becomes most clear by noticing that $Y_i(t)$ is a so called mean-reverting process, i.e.\ as $Y_i(t)$ starts to move away from its long term equilibrium $K$, the drift term starts pulling it back towards $K$ and the speed at which this occurs is given by $\lambda/K$.
Related to this interpretation we find that if we set $\sigma=0$ in expression (\ref{SDE}), we retrieve 
expression (\ref{GI}) with $h(\cdot)=0$ (the GI-process without interaction) 
or equivalently 
the ODE $dY_{i}(t) = \lambda\left(1 - Y_{i}(t)/K\right)dt$.
This ODE is often referred to as the linear growth function (see e.g.\ \cite{RenshawComasMateu,RS2}) and in this setting the parameter $\lambda$ is referred to as the (individual) \emph{growth rate} while the upper bound $K$ often is referred to as the \emph{carrying capacity}.
In conclusion, $\Phi_M$ is controlled by the parameter vector $\theta = (\lambda,K,\sigma,\alpha,\mu)\in\Theta=\Theta_{\lambda}\times\Theta_{K}\times\Theta_{\sigma}\times\Theta_{\alpha}\times\Theta_{\mu}\subseteq\R_{+}^{5}$,
where the pair $(\alpha,\mu)$ controls the time intervals during which the mark functions are non-zero and the remaining parameters control the growth of the marks.

Regarding the initial size $M_i(B_i) = Y_i(0) = M_{i}^{0}$, a few different options are available.
In \cite{Cronie}, the approach was to use the same constant initial value $M_{i}^{0}\equiv M_0\in\R_{+}$ for all individuals in the GI-process,
and in \cite{RS2} the $M_{i}^{0}$'s were chosen as independent $Uni(0,\epsilon)$-distributed random variables, $\epsilon>0$.
Here, however, we also have the further option to sample $M_{i}^{0}$ from the stationary distribution of $Y_i$ (see Section \ref{SectionDistribution} for details).

\subsection{Temporal evolution representation}\label{SectionEvolutionary}

In order to stress that we here consider the temporal evolution of the SG-process (or rather the temporal evolution of the marks), we often write $\Phi_M(t) = (M_1(t),\ldots,M_N(t))$, $t\geq0$.
Furthermore, in order to treat it properly we let it be adapted to some filtered probability space $(\mathcal{X},\mathcal{F},\{\mathcal{F}_t\}_{t\in[0,T)},\P)$. 
Specifically, the family of $\sigma$-algebras $\{\mathcal{F}_t\}_{t\in[0,T]}$ is such that, for any $s \leq t$, $\mathcal{F}_s\subseteq\mathcal{F}_t\subseteq\mathcal{F}$ and, for each $t\in[0,T]$, $\Phi_M(t)$ is $\mathcal{F}_t$-measurable.

We here will construct $\Phi_M(t)$ through two building blocks. 
The first building block is given by the underlying point process $\Phi(t)$, which can be constructed as spatio-temporal point process on $W\times[0,T]$, and we call it a spatial immigration-death (SID) process. 
This process governs the assignment of the spatial locations of the individuals in $W$, as well as their arrival times and their lifetimes. 
The second building block, which may be regarded as an extension of $\Phi(t)$, is the set of $\F$-valued functional marks (stochastic processes). 
We start by describing the underlying SID-process $\Phi(t)$.


Let us consider the SID-process $\{\Phi(t)\}_{t\in[0,T]}$ which is a spatial birth-death process (see e.g.\ \cite{Moller,Kasper}), taking values in the collection $N^f=\{\x\subseteq W:|x|<\infty\}$ of finite point configurations. It has birth rate function $b(\cdot,\cdot)=\alpha$, death rate function $d(\cdot,\cdot)=\mu$ and reference probability measure $\upsilon(B)=\nu(B)/\nu(W)$, $B\in\B(W)$, where $(\alpha,\mu)\in\Theta_{\alpha}\times\Theta_{\mu}\subseteq\R_+^2$. Hence, it 
can easily be verified that the underlying Markov jump process is given by a so-called immigration-death (ID) process ($M/M/\infty$-queue) $\{N(t)\}_{t\in[0,T]}$ (see e.g.\ \cite{CronieYu,GibsonRenshaw,Grimmett}) with arrival rate $\alpha\nu(W)$ and death rate $\mu$. 
Furthermore, we see that the spatial location kernel is such that all locations $X_i\in\Phi(t)\in N^f$, $t\in[0,T]$, are iid $Uni(W)$-distributed.
Looking closer at the ID-process, which is 
a time-homogeneous irreducible positive recurrent Markov chain with state space $\N$, we see that it can be used to describe a population where the ''birth/arrival times'' $B_1<\ldots<B_N$ of the individuals occur according to a Poisson process on $[0,T]$ with intensity $\alpha\nu(W)$ (whereby $N\sim Poi(\alpha\nu(W)T)$) and it generates ''lifetimes'' $L_1,\ldots,L_N$ for the individuals which are iid $Exp(\mu)$-distributed. 
Hereby, by defining $D_i = (B_i+L_i)\wedge T$, $i=1,\ldots,N$, we finalise the equivalence with the construction of the previously defined supports $\mathrm{supp}(M_i)=\{t\in[0,T]:M_i(t)\neq0\} = [B_i,D_i)$, $i=1,\ldots,N$, of the functional marks. 
It is sometimes important to keep track of which individuals are alive/visible and we therefore define the index process $\Omega(t)=\left\{i\in\{1,\ldots,N\}:t\in[B_i,D_i)\right\}$, $\Omega(0)=\emptyset$, which is a Markov process which controls which individuals are alive at time $t$ (note that $N(t) = |\Phi(t)| = |\Omega(t)|$).
We note that we just as well could have defined $\Phi(t)$ as a marked Poisson process on $[0,T]$, with jump times $B_1<\ldots<B_N$ and marks $(L_i,X_i)$, $i=1,\ldots,N$.


We now turn to the second building block of $\Phi_M$. 
Similarly to the previous scenario, the idea here is to consider the stochastic processes $M_i(t) = M_i(t;\Phi) =\1_{[B_i,D_i)}(t)Y_i(t-B_i)$, 
$i=1,\ldots,N$, where the $Y_i(t)$'s are defined in expressions (\ref{SDE}) and (\ref{CIR}). 
Just as before the parameter vector will be given by $\theta = (\lambda,K,\sigma,\alpha,\mu)\in\Theta=\Theta_{\lambda}\times\Theta_{K}\times\Theta_{\sigma}\times\Theta_{\alpha}\times\Theta_{\mu}\subseteq\R_{+}^{5}$. 

We note that under this representation, for each $t\in[0,T]$, we may treat $\Phi_M(t)$ as a (marginal) random vector of a multivariate $d$-dimensional, $d\leq N$, diffusion process with jumps, for which the component processes are independent, stopped and time-shifted CIR-processes with jumps ($d$ is controlled by the supports $[B_i,D_i)$, $i=1,\ldots,N$). 
Note that for the conditional process $\Phi_M(t)|\Phi$, the randomness is present only in the $Y_i(t)$'s. 
Moreover, we note that since $\nu(W)<\infty$ and $T<\infty$, we have that $N<\infty$ a.s..
It is this representation of $\Phi_M$ which mainly will be exploited in the statistical inference parts in the remainder of this paper.

\section{Distributional properties of the SG-process and its components}\label{SectionDistribution}

\subsection{Properties of the CIR-process}
Given below are some results concerning different properties of the CIR-process and they can all be found in e.g.\ \cite{CIR,Iacus}.
The explicit solution of the CIR-process, which is given by
\beann
\label{SolutionCIR}
Y_{i}(t) &=& K - \left(K - M_{i}^{0}\right)\e^{-t\lambda/K} + \sigma\int_{0}^{t}\e^{(s-t)\lambda/K}\sqrt{Y_{i}(s)}dW_i(s),
\eeann
is obtained by applying Ito's formula with $f(x,t) = x\e^{t\lambda/K}$ to the SDE (\ref{SDE}).
Furthermore, when $2\lambda\geq\sigma^2$ the process a.s.\ stays strictly positive whereas it may reach zero otherwise.
This condition, loosely speaking, says that the drift of the SDE must be large enough, in comparison to the diffusion term, to ensure that the mean-reversion is strong enough to keep the process a.s.\ positive.
Hence, we will require that $2\lambda\geq\sigma^2$ so that $M_i(t)>0$ for all $t\in[B_i,D_i)$.

Since $Y_i(t)$ is a Markov process, when we require that $2\lambda\geq\sigma^2$, it is possible to derive explicit statements about the transition distributions, i.e.\ the distributions of the random variables $Y_{i}(t)|Y_{i}(s)$, $s\leq t$.
For instance, when $s<t$ the conditional expectation and variance are given by
\bea
\label{MeanCIR}
\E\left[Y_{i}(t)|Y_{i}(s) = y_{s}\right] &=& K - (K - y_{s})\e^{-(t-s)\lambda/K},\\
\Var\left(Y_{i}(t)|Y_{i}(s) = y_{s}\right)
&=&  y_{s}\frac{\sigma^2 K}{\lambda}\left(\e^{-(t- s)\lambda/K} - \e^{-2(t-s)\lambda/K}\right)\nn\\
&&+ \frac{\sigma^2 K^2}{2\lambda} \left(1 - \e^{-(t-s)\lambda/K}\right)^2,\nn
\eea
respectively. 
More interesting for our purposes, however, is that under the hypothesis that $2\lambda\geq\sigma^2$ and $s\leq t$, conditional on $Y_{i}(s) = y_{s}$, the transition density of $Y_{i}(t)$ is given by the non-central $\chi^2$-distribution density
\bea
\label{TransitionCIR}
p_{Y_{i}}(t-s, y_{t}|y_{s};\lambda,K,\sigma) = a \e^{-(u+v)} \left(\frac{v}{u}\right)^{q/2} I_{q}\left(2\sqrt{uv}\right),
\eea
where $a = 2\lambda/\left(\sigma^2 K\left(1-\e^{-(t-s)\lambda/K}\right)\right)$,
$u = a y_{s}\e^{-(t-s)\lambda/K}$,
$v = a y_{t}$ and
$q = 2\lambda/\sigma^2 - 1$.
The function $I_{q}(x)=\sum_{k=0}^{\infty}(x/2)^{2k+q}/k!\Gamma(k+q+1)$, $x\in\R$, where $\Gamma(\cdot)$ denotes the gamma function, is the modified Bessel function of the first kind of order $q$.

This ergodic process also has a stationary (invariant) distribution $\pi=\pi_{\lambda,K,\sigma}$ which is given by the Gamma distribution with shape parameter $2\lambda/\sigma^2$ and scale parameter $\sigma^2 K/2\lambda$.
Hereby, the density of the stationary distribution is given by
\bea
\label{StatDensity}
\pi(x;\lambda,K,\sigma) = \frac{\left(2\lambda/\sigma^2K\right)^{2\lambda/\sigma^2}}{\Gamma(2\lambda/\sigma^2)} x^{2\lambda/\sigma^2-1}\e^{-x(2\lambda/\sigma^2 K)},\quad x\geq0,
\eea
so that $\pi$ has mean $K$ and variance $\sigma^2 K^2/2\lambda$ and, moreover, for $s<t$, the covariance function of $Y_i(t)$ is given by $\Cov(Y_i(s),Y_i(t)) = \frac{\sigma^2K^2}{2\lambda}\e^{-(t-s)}$.

As previously mentioned, $Y_i(t)$ is a Markov process and given that we start a Markov process in its stationary distribution, it is a strictly stationary process.
In the case of $Y_i(t)$ this means that $Y_i(0) = M_{i}^{0}\sim\pi$ and its finite dimensional distributions (fdds) are shift invariant w.r.t.\ time, i.e.\ $(Y_i(T_1),\ldots,Y_i(T_n))=^d(Y_i(T_1+h),\ldots,Y_i(T_n+h))$ for any set of times $T_1 < \ldots < T_n$, any $h\geq0$ and any $n\in\N$.
Hereby the marginal/transition distributions do not change, i.e.\ for any $(s,t)$, $t>s\geq0$, $Y_i(t)\sim\pi$ and $Y_{i}(t)|Y_{i}(s)\sim\pi$.

\subsection{Properties of the ID-process}

Recall from Section \ref{SectionEvolutionary} the underlying SID-process and its temporal component, the ID-process, $\left\{N(t)\right\}_{t\geq 0}$. 
The following result, which can be found in \cite{CronieYu}, gives us the transition probabilities and the stationary distribution of $N(t)$.

\begin{lem}\label{ImDeTransProb}
The transition probabilities of the ID-process, $N(t)$, are given as convolutions of Poisson densities and Binomial densities such that, for $h,t\geq0$ and $x,y\in\N$,
\beann
\label{ImmigrationDeathTransProb}
p_{N}(t,y|x;\alpha,\mu) &=& \left(P_{Poi(\rho)} \ast P_{Bin(x,\e^{-\mu t})}\right)\big(y\big) \\
												&=& \sum_{k=0}^{y} P_{Poi(\rho)}(k) P_{Bin(x,\e^{-\mu t})}(y-k),
\eeann
where $P_{Poi(\rho)}(\cdot)$ is the Poisson density with parameter $\rho=\alpha\left(1-\e^{-\mu t}\right)/\mu$ and $P_{Bin(x,\e^{-\mu t})}(\cdot)$ is the Binomial density with parameters $x$ and $\e^{-\mu t}$.

Furthermore, the stationary distribution of $N(t)$ is given by $$\pi_{N}(\cdot)=\P(Poi(\alpha/\mu)\in\cdot)$$ and the expected value and second moment of the $p_{N}(t,y|x;\alpha,\mu)$-distribution are given by $\E[N(h+t)|N(h)=i]=i\e^{-\mu t} + \rho$
and $\E[N^2(h+t)|N(h)=i] = i(i-1)\e^{-2\mu t} + (1+2\rho)i\e^{-\mu t} + \rho^2 + \rho$, respectively.
\end{lem}

This lemma will be further exploited in Proposition \ref{fidi}, where the fdds of $\Phi_M(t)$ are derived.

\subsection{Finite dimensional distributions of the SG-process}

Consider now the SID-process $\Phi(t)$, or alternatively the index process $\Omega(t)$ and the population size process $N(t)=|\Phi(t)|=|\Omega(t)|$.
Recall that these processes as well as the CIR-process are Markov processes, which in turn implies that also $\Phi_M(t)$ is a Markov process.
This observation will be of great importance since in Proposition \ref{fidi} the Markov property we will be exploited in the derivation of the fdds of $\Phi_M(t)$.

In order to set the framework, we consider the (sample) times $0 = T_0 < T_1 < \ldots < T_n \leq T$ and the distribution of $\left(\Phi_{M}(T_1),\ldots,\Phi_{M}(T_n)\right)^T$, when we are concerned with exactly, say, $d\in\{1,\ldots,N\}$ individuals who appear at $T_1,\ldots,T_n$ (recall that $N$ is the total number of individuals observed if we monitor the process continuously).
Furthermore, provided that the joint density of $\left(\Phi_{M}(T_1),\ldots,\Phi_{M}(T_n)\right)^T$ exists,
when evaluated at the size-time matrix
\beann
\mathbf{M} =
\begin{pmatrix}
m_{11}	&	\cdots	&	m_{1n}\\
\vdots	&	\ddots	&	\vdots\\
m_{d1}	&	\cdots	&	m_{dn}
\end{pmatrix}
\in\R^{d\times n},
\eeann
we will denote this density by $\mathbf{p}_{T_1,\ldots,T_n}(\mathbf{M};\theta)$.
It should be emphasised that the $i$th row of $\mathbf{M}$ represents the evaluation-sizes of the $i$th individual under consideration, at the respective times $T_1,\ldots,T_n$.
We further also note that if $m_{ik} = 0$, we are considering the case where the $i$th individual is not alive at time $T_k$.
Hence, if $m_{ik} = 0$ for all $k=1,\ldots,l-1<n$, and $m_{il} > 0$, we evaluate a scenario where $B_i\in(T_{l-1},T_{l}]$, and when $m_{il} > 0$ and $m_{ik} = 0$ for all $k=l+1,\ldots,n$ we consider $D_i\in(T_{l},T_{l+1}]$. 
Consequently, if a row were to contain only zeros, we would be considering an individual who is not alive at any of $T_1,\ldots,T_n$, whence that individual/row may be removed from consideration.

The exact form of $\mathbf{p}_{T_1,\ldots,T_n}(\mathbf{M};\theta)$ is given in Proposition \ref{fidi} and the main feature exploited in its derivation is the Markovianity of $\Phi_M(t)$.
We note that the distribution of $\left(\Phi_{M}(T_1),\ldots,\Phi_{M}(T_n)\right)^T$ may be expressed through $\Phi_M(t)$'s transition probabilities/densities, which are given by
\bea
\label{TransSDE}
\P\left(\Phi_{M}(t)\in\mathbf{A}|\mathcal{F}_s;\theta\right) &=& \P\left(\Phi_{M}(t)\in\mathbf{A}|\Phi_{M}(s); \theta\right),
\eea
where $\mathbf{A} = A_1 \times\ldots\times A_d\in\mathcal{B}(\R^{d})$ and $0 \leq s < t \leq T$.
The proof of Proposition \ref{fidi} can be found in the Appendix.

\begin{prop}[Fdds of $\Phi_{M}(t)$]\label{fidi}
Given $0 = T_0 < T_1 < \ldots < T_n \leq T$ and $\Phi_{M}(T_0)$, if we let $M_i^0=M_0>0$ for all $i$, then the joint density of $\left(\Phi_{M}(T_1),\ldots,\Phi_{M}(T_n)\right)^T$, evaluated at $\mathbf{M}\in\R^{d\times n}$, $d\geq1$, is given by
\bea
\label{JointDensity}
\mathbf{p}_{T_1,\ldots,T_n}(\mathbf{M};\theta)
&=& C
\prod_{k=1}^{n} p_{N}\left(\Delta T_{k}, |\omega_k| \Big| |\omega_{k-1}|;\alpha\nu(W),\mu\right)\\
&&\times \prod_{k=1}^{n} \prod_{i\in\omega_{k-1}\cap\omega_{k}} p_{Y_{1}}(\Delta T_{k}, m_{ik}|m_{i(k-1)};\lambda,K,\sigma)\nn\\
&&\times
\prod_{i=1}^{d}
\int_{T_{k_{i}-1}}^{T_{k_{i}}} \frac{p_{Y_{1}}(T_{k_{i}}-t, m_{i(k_{i}-1)}|M_0;\lambda,K,\sigma)}{T_{k_{i}}-T_{k_{i}-1}} dt,\nn
\eea
where $\Delta T_{k} = T_{k}-T_{k-1}$, $\omega_k = \{i:m_{ik} > 0\}$, $k=1,\ldots,n$, and $k_{i}=\min\{k:i\in\omega_k\}$, $i=1,\ldots,d$.
The constant $C = C(\nu(W),\mathbf{M})>0$ can be found in expression (\ref{NormalizingConstant}), and the densities $p_{Y_{1}}(\cdot)$ and $p_{N}(\cdot)$ are given, respectively, by expression (\ref{TransitionCIR}) and Lemma \ref{ImDeTransProb}.
\end{prop}

Conditioning on $\Phi_{M}(T_0)$ is reasonable since we in most applications already have all the information about the marked points present at the first sample time point.
Note that if we choose all $M_i^0$ fixed but not necessarily equal, expression (\ref{JointDensity}) only changes in that $M_i^0$ replaces $M_0$.
Furthermore, from the proof of Proposition \ref{fidi} we see that the transition probabilities (\ref{TransSDE}) are obtained by finding 
\beann
\P\left(\Phi_{M}(t)\in\mathbf{A}|\Phi_{M}(s)=\y; \theta\right)
= \int_{\mathbf{A}} \mathbf{p}_{\Phi_{M}(t)|\Phi_{M}(s)}(\m|\y;\theta)d\m, 
\eeann
where $\mathbf{A} = A_1 \times\ldots\times A_d\in\mathcal{B}(\R^{d})$, $\m=(m_1,\ldots,m_d)^T\in\R^{d}$, $\y=(y_1,\ldots,y_d)^T\in\R^{d}$, 
\beann
\mathbf{p}_{\Phi_{M}(t)|\Phi_{M}(s)}(\m|\y;\theta) &=& 
C(\nu(W),\m,\y)\;
p_{N}\left(t-s, |\omega(\m)| \Big| |\omega(\y)|; \alpha\nu(W),\mu\right)\\
&&\times
\prod_{i\in\omega(\y)\cap\omega(\m)} p_{Y_{1}}(t-s, m_{i}|y_{i};\lambda,K,\sigma)\\
&&\times
\prod_{i\in\omega(\y)^c\cap\omega(\m)} \frac{1}{t-s} \int_{s}^{t} p_{Y_{1}}(t-v, m_{i}|M_0;\lambda,K,\sigma) dv,
\eeann
$\omega(\m) = \{i:m_i>0\}$, $\omega(\y) = \{i:y_i>0\}$ and $C(\nu(W),\m,\y)$ is a constant.

As mentioned before, when $M_{i}^{0}\sim\pi$ we have that $Y_i(t)$ is a strictly stationary process and this will have a further impact on the joint densities in Proposition \ref{fidi}.
\begin{cor}
Given the preliminaries and notation of Proposition \ref{fidi}, by instead assuming that $M_{i}^{0}\sim\pi$, the joint density (\ref{JointDensity}) becomes
\bea
\label{StatJointDensity}
\mathbf{p}_{T_1,\ldots,T_n}(\mathbf{M};\theta)
&=& 
C
\prod_{k=1}^{n} p_{N}\left(\Delta T_{k},|\omega_k| \Big| |\omega_{k-1}|;\alpha\nu(W),\mu\right)\nn\\
&&\times \prod_{k=1}^{n} \prod_{i\in\omega_{k}} \pi(m_{ik};\lambda,K,\sigma).
\eea
\end{cor}

\begin{proof}
We note that all transition densities $p_{Y_{i}}(\Delta T_{k}, \cdot|\cdot;\lambda,K,\sigma)$ in expression (\ref{MarkJointDensity}) may be replaced by the stationary Gamma densities $\pi(\cdot;\lambda,K,\sigma)$ of expression (\ref{StatDensity}).
\end{proof}

\begin{rem}
We may additionally require that also $N(t)$ starts in its stationary distribution $\pi_{N}$ (see Lemma \ref{ImDeTransProb}) so that also $N(t)$ becomes a strictly stationary process.
Hereby the transition probabilities $p_{N}(\cdot)$ in expression (\ref{StatJointDensity}) will be replaced by $\pi_{N}\left(|\omega_k|;\alpha\nu(W),\mu\right)$.
Note that this change will imply that $N(t)=|\Omega(t)|\sim Poi(\alpha/\mu)$ for all $t\geq0$ and under this setup, since all $Y_i$'s are stationary, we have that $M_{i}(0)\sim\pi$ for all individuals $i\in\Omega(0)$.
\end{rem}

\begin{rem}
As we previously noted, conditionally on $\Omega(0)=\emptyset$, the process $\Xi(t) = \bigcup_{i\in\Omega(t)}B_{X_i}[M_i(t)]$ at each fixed time $t$ corresponds to a Boolean model (see e.g.\ \cite{StoyanKendallMecke}). The \emph{germs} $\{X_i\}_{i\in\Omega(t)}$ are generated from a Poisson process with intensity measure $\Lambda_t(B) = \frac{\alpha}{\mu}(1-\e^{-\mu t})\nu(B\cap W)$, $B\in\B(\R^2)$, and the \emph{grains} are given by $\{B_{X_i}[M_i(t)]\}_{i\in\Omega(t)}$, where all $M_i(t)$'s are iid $\Gamma(2\lambda/\sigma^2,\sigma^2 K/2\lambda)$-distributed.
Note that this follows since $\Omega(t)$ can be generated as a thinned Poisson process (see \cite{CronieYu}).
\end{rem}

\section{Maximum likelihood estimation}\label{SectionLikelihood}

Conditionally on $\Phi_{M}(T_0)=\Phi_{M}(0)$, we now assume that we sample the SG-process $\Phi_{M}(t)$ as $M=(\phi_1,\ldots,\phi_n)$ at the sample times $T_1,\ldots,T_n$ on some region $W$.
Here $\phi_k = \left(m_{1k},\ldots,m_{dk}\right)^T$,
$k=1,\ldots,n$, where $d=|\bigcup_{k=1}^{n}\omega_k|$ and $\omega_k = \{\text{indices of individuals present at time } T_k\} = \{i:m_{ik}>0\}$ 
(we may write $\phi_k = \left(\1_{\omega_k}(1)m_{1k},\ldots,\1_{\omega_k}(N)m_{dk}\right)^T$ to emphasise the individuals' life status). 
Now, based on this sampling scheme we want to find the Maximum Likelihood (ML) estimate of the parameter vector $\theta = (\lambda,K,\sigma,\alpha,\mu)\in\Theta$.

We note that when $\Phi_{M}$ is treated as in Section \ref{SectionFMPP}, i.e.\ as a functional marked point process instead of as an evolutionary process, the estimation based on the current sampling is equivalent to estimating a thinned version of the process. 
More specifically, this thinned version is such that all marked points $A=\{[X_i,M_i]:[B_i,D_i)\cap\{T_1,\ldots,T_n\}=\emptyset\}$ are removed and only the partial information $\{(X_i,M_i(T_1),\ldots,M_i(T_n)):i\notin A\}$ is available to estimate the actual structure of the non-thinned process $\Phi_{M}$ (think of this as a sample from the previously mentioned Boolean model which was based on solely the ''alive individuals'').

The likelihood function of the parameters of the SG-process, $\L_n(\theta)=\L_n(\theta;M)$, is given by the joint density of $\left(\Phi_{M}(T_1),\ldots,\Phi_{M}(T_n)\right)$, evaluated at $M=\left(\phi_1,\ldots,\phi_n\right)$ and treated as a function of $\theta\in\Theta$.
Therefore, depending on whether we choose $M_{i}(0)$ to be fixed or drawn from the stationary distribution, we end up evaluating either expression (\ref{JointDensity}) or expression (\ref{StatJointDensity}) when we evaluate $\L_n(\theta)$.

\subsection{ML-estimation: $M_{i}^{0}=M_0\in\R_+$}\label{SectionNonStationaryMLE}
When we let all $Y_{i}(0)=M_i^0=M_0\in\R_+$ be given by the same fixed value, from expression (\ref{JointDensity}) we obtain
\bea
\label{FullLikelihood}
\L_n(\theta) =
C
\L_{1,n}(\theta) \L_{2,n}(\theta) \L_{3,n}(\theta)
\propto \L_{1,n}(\theta) \L_{2,n}(\theta) \L_{3,n}(\theta),
\eea
where, for $k_{i}=\min\{k:i\in\omega_k\}$ and $\Delta T_{k} = T_{k} - T_{k-1}$, $k=1,\ldots,n$,
\beann
\L_{1,n}(\theta)	&=&	\prod_{k=1}^{n} \prod_{i\in\omega_{k-1}\cap\omega_{k}} p_{Y_{1}}(\Delta T_{k}, m_{ik}|m_{i(k-1)};\lambda,K,\sigma)\\
\L_{2,n}(\theta)	&=&	\prod_{i\in\bigcup_{k=1}^{n}\omega_k} \frac{1}{\Delta T_{k_{i}}} \int_{0}^{\Delta T_{k_{i}}} p_{Y_{1}}(t, m_{i(k_{i}-1)}|M_0;\lambda,K,\sigma) dt\\
\L_{3,n}(\theta)	&=&	\prod_{k=1}^{n} p_{N}\left(\Delta T_{k},|\omega_k| \Big| |\omega_{k-1}|;\alpha\nu(W),\mu\right).
\eeann
The (rescaled) log-likelihood is given by
\beann
\label{FullLogLikelihood}
l_n(\theta) &=&
\log\left(C^{-1}\L_n(\theta)\right)
= \log\L_{1,n}(\theta) + \log\L_{2,n}(\theta) + \log\L_{3,n}(\theta)\nn\\
&=:& l_{1,n}(\theta) + l_{2,n}(\theta) + l_{3,n}(\theta),
\eeann
whereby the ML-estimator of $\theta\in\Theta$, based on $\left(\Phi_{M}(T_1),\ldots,\Phi_{M}(T_n)\right)$, will be given by
\beann
\widetilde{\theta}_n &:=& \widetilde{\theta}_n\left(\Phi_{M}(T_1),\ldots,\Phi_{M}(T_n)\right)\\
&=& \argmax_{\theta\in\Theta} l_n(\theta;\Phi_{M}(T_1),\ldots,\Phi_{M}(T_n))\\
&=& \argmax_{\theta\in\Theta} (l_{1,n}(\theta) + l_{2,n}(\theta) + l_{3,n}(\theta)).
\eeann
We now want to express the ML-estimator $\widetilde{\theta}_n = (\widetilde{\lambda}_n, \widetilde{K}_n, \widetilde{\sigma}_n, \widetilde{\alpha}_n, \widetilde{\mu}_n)$ as the sum of two estimators $\widetilde{\theta}_{1,n}$ and $\widetilde{\theta}_{2,n}$ which, respectively, handle the separate estimation of $(\lambda,K,\sigma)$ and $(\alpha,\mu)$.
We note that $l_{1,n}(\theta) + l_{2,n}(\theta)$, which only involves $\lambda$, $K$ and $\sigma$, will be maximized by any $\widetilde{\theta}_{1,n} = (\widetilde{\lambda}_n,\widetilde{K}_n,\widetilde{\sigma}_n,\alpha,\mu)$, $(\alpha,\mu)\in\R^2$.
Similarly we have that $l_{3,n}(\theta)$, which only involves $\alpha$ and $\mu$, will be maximized by $\widetilde{\theta}_{2,n} = (\lambda,K,\sigma,\widetilde{\alpha}_n,\widetilde{\mu}_n)$, for any $(\lambda,K,\sigma)\in\R^3$.
Hence, in order for $\widetilde{\theta}_{n} = \widetilde{\theta}_{1,n} + \widetilde{\theta}_{2,n}$ to hold, we must require that $\widetilde{\theta}_{1,n} = (\widetilde{\lambda}_n,\widetilde{K}_n,\widetilde{\sigma}_n,0,0)$ and $\widetilde{\theta}_{2,n} = (0,0,0,\widetilde{\alpha}_n,\widetilde{\mu}_n)$, i.e.\
\bea
\label{MLE}
\widetilde{\theta}_n
&=& \widetilde{\theta}_{1,n} + \widetilde{\theta}_{2,n}\\
&=& \argmax_{\theta\in\Theta_{\lambda}\times\Theta_{K}\times\Theta_{\sigma}\times\{0\}^2} \left\{l_{1,n}(\theta) + l_{2,n}(\theta)\right\}
+ \argmax_{\theta\in\{0\}^3\times\Theta_{\alpha}\times\Theta_{\mu}} l_{3,n}(\theta), \nn
\eea
and consequently we may estimate the parameters of the ID-process and the parameters related to the mark growth separately.

When the amount of data is large or when the $\Delta T_k$'s are small, we may consider instead the approximate ML-estimation where we set $l_{2,n}(\theta)=0$ so that the only information about the diffusions comes from the observed transitions.
This is reasonable since the amount of information about the actual parameter values which is carried by $l_{2,n}(\theta)$ is not really substantial (in comparison to $l_{1,n}(\theta)$).
Moreover, since there is no closed form expression available for the ML-estimator $(\widetilde{\alpha}_n,\widetilde{\mu}_n)$ of the ID-process (see \cite{CronieYu}), there is also no closed form available for $\widetilde{\theta}_n$ in (\ref{MLE}).
Hence, in modelling situations one has to rely on numerical methods to find $\widetilde{\theta}_n$.

\subsection{ML-estimation: $M_{i}^{0}\sim\pi$}\label{SectionStationaryMLE}

Under the assumption that we start the diffusions in their stationary distributions, $M_{i}^{0}\sim\pi$, from expression (\ref{StatJointDensity}) we obtain the likelihood function
\bea
\label{StatFullLikelihood}
\L_n(\theta) = C \L_{1,n}(\theta) \L_{2,n}(\theta)
\propto
\L_{1,n}(\theta) \L_{2,n}(\theta)
\eea
and the (rescaled) log-likelihood
\beann
\label{FullLogLikelihood}
l_n(\theta) &=& \log\left(C^{-1}\L_n(\theta)\right)
= \log\L_{1,n}(\theta) + \log\L_{2,n}(\theta)\nn\\
&=:&
l_{1,n}(\theta) + l_{2,n}(\theta),\qquad
\eeann
where, for $\Delta T_{k} = T_{k} - T_{k-1}$, $k=1,\ldots,n$,
\beann
l_{1,n}(\theta)	&=&	\log\left(\prod_{k=1}^{n} \prod_{i\in\omega_{k}} \pi(m_{ik};\lambda,K,\sigma)\right)
=\sum_{k=1}^{n} \sum_{i\in\omega_{k}} \log\pi(m_{ik};\lambda,K,\sigma)\\
l_{2,n}(\theta)	&=& \log\left(\prod_{k=1}^{n} p_{N}\left(\Delta T_{k},|\omega_k| \Big| |\omega_{k-1}|;\alpha\nu(W),\mu\right)\right)\\
&=& \sum_{k=1}^{n} \log p_{N}\left(\Delta T_{k},|\omega_k| \Big| |\omega_{k-1}|;\alpha\nu(W),\mu\right).
\eeann
Here, just as in the fixed initial value case of Section \ref{SectionNonStationaryMLE}, we deal with the separate estimators
\bea
\label{StatMLE}
\hat{\theta}_n &=& \hat{\theta}_{1,n} + \hat{\theta}_{2,n}\\
&=& \argmax_{\theta\in\Theta_{\lambda}\times\Theta_{K}\times\Theta_{\sigma}\times\{0\}^2} l_{1,n}(\theta)
+ \argmax_{\theta\in\{0\}^3\times\Theta_{\alpha}\times\Theta_{\mu}} l_{2,n}(\theta)\nn
\eea
and, similarly, there is no closed form expression available for $\hat{\theta}_n$.

\section{Asymptotic inference under stationarity}\label{SectionAsymptotics}
When dealing with asymptotic spatial statistics, there are different types of asymptotics which may be considered.

In the case of the SG-process, within the framework of so called increasing domain asymptotics (see e.g.\ \cite{Zhang}), there are essentially two different ways to increase the total number of individuals observed, and consequently also the number of transitions taking place between pairs of consecutive sample times $T_{k-1}$ and $T_{k}$.
The first approach is to increase the number of sample times of the mark processes, i.e.\ we let $n$ grow, whereby $T_n=T$ also will grow.
The second approach is to gradually increase the size of the sampling window $W$ (with the number of sample times fixed).
The two approaches are similar since in both cases we increase the parameter of the Poisson distribution of $N\sim Poi(\alpha T\nu(W))$.
Here, we choose to consider only the first of the two alternatives.

Consider the situation where we, without loss of generality, let $W=[0,1]^2$ and apply the equidistant sampling scheme $T_k = k\Delta$, $k=1,\ldots,n$, $\Delta>0$, where $T = T_n = n\Delta$.
In what follows we denote by $\theta_0=(\lambda_0,K_0,\sigma_0,\alpha_0,\mu_0)\in\Theta$ the true/underlying parameter vector which is responsible for generating $\Phi_M$ and we assume that $\Theta$ is a subset of $\R_{+}^5$ such that
\bea
\label{ParamSpace}
\Theta\cap\{(\lambda,K,\sigma,\alpha,\mu)\in\R_+^5:2\lambda<\sigma^2\} = \emptyset.
\eea
Recall that this is required to keep the $Y_i(t)$'s positive.

In the theorems and corollaries below we give the strong consistency and the asymptotic normality of the ML-estimator.
The proofs are given in the Appendix. 
The consistency proof follows the approach suggested by Wald \cite{Wald} and the asymptotic normality follows the lines of the classical approach of Cram\'er (see e.g.\ \cite{Ferguson}).

\begin{thm}[Consistency]\label{ThmCons}
Let $\Theta$ be a compact subset of $\R_+^5$ such that (\ref{ParamSpace}) holds.
Then, for $\theta_0\in\Theta$, the estimator $\hat{\theta}_n$ in expression (\ref{StatMLE}) is strongly consistent, i.e.\ as $n\rightarrow\infty$,
\begin{center}
$\hat{\theta}_n \stackrel{a.s.}{\longrightarrow} \theta_0$.
\end{center}
\end{thm}

Now, by putting some additional restrictions on the parameters we may also prove the following theorem.

\begin{thm}[Asymptotic normality]\label{ThmAsNorm}
Let $\theta_0$ be an interior point of $\Theta$, where $\Theta$ is a compact subset of $\R_+^5$ such that (\ref{ParamSpace}) holds.
Require further that $\theta_0$ and $\Delta>0$ are such that $(\log(\alpha_0 + \mu_0) - \log(\alpha_0))/\mu_0 \geq 2\Delta$.

Assume that $\lambda_0$ is known, so that $\hat{\theta}_n = (\hat{K}_n, \hat{\sigma}_n, \hat{\alpha}_n, \hat{\mu}_n)$ is the ML-estimator of $\theta_0 = (K_0,\sigma_0,\alpha_0,\mu_0)$.
Then, as $n\rightarrow\infty$, we obtain
\beann
\sqrt{n}\big(\hat{\theta}_{n} - \theta_0\big)
\stackrel{d}{\longrightarrow}
\mathbf{Y}\sim
N\left(\mathbf{0}_{4\times1},
\begin{bmatrix}
\frac{\mu_0}{\alpha_0}\frac{K_{0}^2\sigma_{0}^2}{2\lambda_{0}}	&	0  & \mathbf{0}_{1\times2}\\
0 & \frac{\mu_0}{\alpha_0}\frac{\sigma_{0}^4}{8\lambda_{0}C(\theta_0)}	& \mathbf{0}_{1\times2}\\
\mathbf{0}_{2\times1}	&	 \mathbf{0}_{2\times1}	 &	I_N(\theta_0)^{-1}
\end{bmatrix}
\right),
\eeann
where $C(\theta) = \frac{2\lambda}{\sigma^2}\psi'\left(\frac{2\lambda}{\sigma^2}\right) - 1 > 0$, $\psi(x) = \Gamma'(x)/\Gamma(x)$, $\Gamma(\cdot)$ is the gamma function, 
$\mathbf{0}_{i\times j}$ denotes the $i\times j$ zero matrix and the $2\times2$ matrix $I_N(\theta_0)^{-1}$, which can be found in expression (\ref{InvFisherImDea}), is the covariance matrix related to the ID-process.

Similarly, when $\sigma_0$ is known, we estimate $\theta_0 = (\lambda_0,K_0,\alpha_0,\mu_0)$
by means of the ML-estimator $\hat{\theta}_n = (\hat{\lambda}_n, \hat{K}_n, \hat{\alpha}_n, \hat{\mu}_n)$ and, as $n\rightarrow\infty$, we obtain
\beann
\sqrt{n}\big(\hat{\theta}_{n} - \theta_0\big)
\stackrel{d}{\longrightarrow}
\mathbf{Y}\sim
N\left(\mathbf{0}_{4\times1},
\begin{bmatrix}
\frac{\mu_0}{\alpha_0}\frac{\lambda_{0}\sigma_{0}^2}{2 C(\theta_0)}	&	0  & \mathbf{0}_{1\times2}\\
0  & \frac{\mu_0}{\alpha_0}\frac{K_{0}^2\sigma_{0}^2}{2\lambda_{0}}	& \mathbf{0}_{1\times2}\\
\mathbf{0}_{2\times1}	&	\mathbf{0}_{2\times1}	&	 I_N(\theta_0)^{-1}
\end{bmatrix}
\right).
\eeann

The Fisher information for the discretely sampled ID-process is given by
\bea
\label{FisherImDea}
I_N(\theta_0) &=&
\begin{pmatrix}
I_N(\theta_0)_{11}	& I_N(\theta_0)_{12}\\
I_N(\theta_0)_{12}	&	I_N(\theta_0)_{22}
\end{pmatrix},
\eea
where
\beann
I_N(\theta_0)_{11} &=&	(\Xi-1)\frac{\rho_{0}^2}{\alpha_{0}^2},
\qquad
I_N(\theta_0)_{12} 	=		\frac{(\Xi-1)\rho_{0}(\mu_0\Delta - \tau_{0}) - \mu_{0}\Delta}{\mu^2_{0}},\\
I_N(\theta_0)_{22} &=&	\frac{\alpha_0^2 \mu_0 \Delta(2\tau_0 -\mu_0t)}{\rho_0 \mu_0^4} + \frac{\alpha_0^2 \Delta^2 \e^{-\mu_0 \Delta}}{\mu_0^2 \rho_0}
+ \frac{(\Xi-1)\alpha_{0}^2\left(\tau_{0} - \mu_{0}\Delta\right)^2}{\mu_{0}^4},
\eeann
$\Xi = \sum_{i, j\in\N} \frac{p_{N}\left(\Delta,j-1|i;\alpha_0,\mu_0\right)^2}{p_{N}\left(\Delta,j|i;\alpha_0,\mu_0\right)} \pi_{N}(i;\alpha_0,\mu_0)$,
$\tau_0 = 1 - \e^{-\mu_0 \Delta} - \mu_0 \Delta\e^{-\mu_0 \Delta}$, 
$\rho_{0} = \frac{\alpha_0}{\mu_0}(1 - e^{-\Delta\mu_0})$, 
and its inverse is given by
\begin{equation}
\label{InvFisherImDea}
\begin{array}{l}
I_N(\theta_0)^{-1} = \frac{\mu_0}{\Delta\left(\left(1 + \e^{-\mu_0 \Delta}\right)\rho_0(\Xi - 1) - 1\right)}\\
\times
\begin{pmatrix}
\frac{\rho_0\left(2\tau_0 - \mu_0 \Delta\left(1 - \e^{-\mu_0 \Delta}\right)\right) + \frac{\rho_0^2}{\mu_0 \Delta}(\Xi - 1)(\tau_0 - \mu_0 \Delta)^2}{\left(1 - \e^{-\mu_0 \Delta}\right)^2}
& 1 + \frac{\rho_0(\Xi - 1)(\tau_0 - \mu_0 \Delta)}{\mu_0 \Delta} \\
1 + \frac{\rho_0(\Xi - 1)(\tau_0 - \mu_0 \Delta)}{\mu_0 \Delta}
& \frac{(\Xi - 1)\left(1 - \e^{-\mu_0 \Delta}\right)^2}{\mu_0 \Delta}
\end{pmatrix}.
\end{array}
\end{equation}

\end{thm}

The reason that we require knowledge of either $\lambda_0$ or $\sigma_0$ in Theorem \ref{ThmAsNorm} is related to the over parametrization of the $\Gamma(\beta_1,\beta_2)$-distribution, $\beta_1 = 2\lambda/\sigma^2$, $\beta_2 = \sigma^2 K/2\lambda$.
We note to begin with that $\beta_2 = K/\beta_1$ and for a random variable $Z\sim\Gamma(\beta_1,\beta_2)$, by consulting expression (\ref{2orderPartialDeriavtives}) in the Appendix, we obtain the related positive semi-definite singular (non-invertible) Fisher information
\beann
\label{FisherGamma}
I_Z(\theta) = -\E_{\theta}\left[\frac{\partial^2\log\pi(Z;\lambda,K,\sigma)}{\partial(\lambda,K,\sigma)^2}\right]
=
\frac{2 C(\theta)}{\sigma^2}
\begin{pmatrix}
\frac{1}{\lambda}
& 0
& \frac{-2}{\sigma}\\
  0
& \frac{\lambda}{K^2}C(\theta)^{-1}
& 0 \\
  \frac{-2}{\sigma}
& 0
& \frac{4\lambda}{\sigma^2}
\end{pmatrix}.
\eeann

\begin{rem}
From the proofs of Theorem \ref{ThmCons} and Theorem \ref{ThmAsNorm} it may be seen that if we reduce $l_{1,n}(\theta)$ to
\beann
\widetilde{l}_{1,n}(\theta)	&=&	\sum_{k=1}^{n} \sum_{i\in\widetilde{\omega}_{k}} \log\pi(m_{ik};\lambda,K,\sigma),
\eeann
where $\widetilde{\omega}_{k}\subseteq\omega_{k}$ ($\emptyset=\widetilde{\omega}_{k}$ iff $\omega_{k}=\emptyset$), the proofs of the consistency and the asymptotic normality still go through (with obvious modifications).
However, the convergence speed will be different as well as the Fisher information $I(\theta_0)$.
An example of such a reduction is to choose $\widetilde{\omega}_{k} = \{\omega_{k,1}\}$, i.e.\ we choose just one element from $\omega_{k}$.
Another example of a reduction under which the results still hold is to consider the subsequence $k_n=n\wedge A(n)$, $A(n)=\sum_{k=1}^{n}|\omega_{k}|$, and the reduction
\beann
\widetilde{l}_{1,n}(\theta)	&=&	\sum_{i=1}^{k_n} \log\pi(m_{ik};\lambda,K,\sigma).
\eeann
\end{rem}

\section{Evaluation of the estimators}\label{SectionEvaluation}
We now turn to the numerical evaluation of our ML-estimators and precisely we are interested in investigating the asymptotic robustness of the stationarity assumption. 
This is carried out by assuming that the data is generated with some fixed $M_i^0=M_0$ and some $\theta_0$, while we instead are employing the estimator $\hat{\theta}_n$ in expression (\ref{StatMLE}) of Section \ref{SectionStationaryMLE}, i.e.\ the estimator based on the assumption that $M_i^0\sim\pi$, to estimate $\theta_0$. 
We then compare the behaviours of $|\hat{\theta}_n-\theta_0|$ and $|\widetilde{\theta}_n-\theta_0|$.

We first note that we from expression (\ref{MeanCIR}) may conclude that $\e^{-t\lambda_0/K_0} = |K_0 - \E[Y_{i}(t)|Y_{i}(0) = M_i^0]|/|K_0 - M_i^0|$.
Clearly, if $|K^0 - M_i^0|$ is small then $Y_i(t)$ quickly approaches its steady state $K_0$, whence the distance $|\widetilde{\theta}_n-\hat{\theta}_n|$ between the two estimators should be reduced.
The same should hold if (additionally) $\lambda_0$ is large, since under this condition the mean reversion is strong, which results in small deviations from the long term mean $K_0$.
Similarly, if $\sigma_0$ is small then the random fluctuations do not influence the growth as much as the drift coefficient $\lambda_0(1-Y_i(t)/K_0)$ and hereby the drift becomes the main determining factor of the speed of convergence to $K_0$.
We also note that if $\mu_0$ is small then the expected lifetime of an individual, $\E_{\theta_0}[L_i] = 1/\mu_0$, tends to be longer whereby we obtain more samples of $Y_i(t)$ when it is close to its steady state $K_0$.

We simulate $30$ trajectories of $\Phi_M(t)$ on $W=[0,1]^2$ and sample them discretely according to the sampling scheme $T_k=k$, $k=1,\ldots,100$.
Then, by using the stationary ML-estimator $\hat{\theta}_n$ of Section \ref{SectionStationaryMLE}, we reestimate the parameters and compare the behaviour of $\hat{\theta}_n$ with the (non-stationary) ML-estimator $\widetilde{\theta}_n$ of Section \ref{SectionNonStationaryMLE}.
We use different values for the parameters $\theta_0$ and $M_i^0$ to assess when $|\hat{\theta}_n-\theta_0|$ and $|\widetilde{\theta}_n-\theta_0|$ are small.
We here only consider the estimation of $\lambda_0$, $K_0$ and $\sigma_0$ since the performance of $(\hat{\alpha}_n,\hat{\mu}_n)$ already has been evaluated in \cite{CronieYu}.

As we can see in Table \ref{TableEvaluation}, as expected, in the case of $\hat{\theta}_n$ the main determining factor of the bias is the size of $\lambda_0$, although the size of $|K-M_i^0|$ certainly plays a role.
It should also be noted that a higher $\sigma_0$ seems to imply a lower bias for $\hat{\sigma}_n$.
Furthermore, we see that $\widetilde{\theta}_n$ outperforms $\hat{\theta}_n$ in each case given in Table \ref{TableEvaluation}, which is to be expected since $\widetilde{\theta}_n$ is the estimator which is based on the correct model assumption.
Note, however, that there are parameter choices for which even the performance of $\widetilde{\theta}_n$ is a bit poor.

\begin{table}[htbp]
	\centering
		\begin{tabular}{lccclccc}
		\hline
		$M_i^0 = 0.1$								&	$\lambda$	&	$K$			&	$\sigma$	&	$M_i^0 = 0.1$						&	$\lambda$	&	$K$			&	$\sigma$\\
		\hline
		True ($\theta_0$)						&	0.5				&	5				&	0.1				&	 True ($\theta_0$)				&	0.5				&	5				&	0.1			\\
		Mean $\widetilde{\theta}_n$	& 0.5027		&	5.1698	&	0.1086		&	Mean $\hat{\theta}_n$		&	 3.2856		&	3.9807	&	0.9761	\\
		Bias $\widetilde{\theta}_n$	& 0.5\%			&	3.4\%		&	8.6\%			&	Bias $\hat{\theta}_n$		 &	557.1\%		&	-20.4\%	&	876.1\%	\\
		S.e. $\widetilde{\theta}_n$	&	0.0605		&	0.5623	&	0.0139		&	S.e. $\hat{\theta}_n$	&	 1.3928		&	0.1266	&	0.2480		\\
		\hline
		$M_i^0 = 5$									&	$\lambda$	&	$K$			&	$\sigma$	&	$M_i^0 = 5$							&	$\lambda$	&	$K$			&	$\sigma$\\
		\hline
		True ($\theta_0$)						&	0.5				&	5				&	0.1				&	 True ($\theta_0$)				&	0.5				&	5				&	0.1			\\
		Mean $\widetilde{\theta}_n$	& 0.4241		&	5.0385	&	0.1006		&	Mean $\hat{\theta}_n$		&	 2.9054		&	4.9987	&	0.2185	\\
		Bias $\widetilde{\theta}_n$	& -15.2\%		&	0.8\%		&	0.6\%			&	Bias $\hat{\theta}_n$		 &	481.1\%		&	-0.03\%	&	118.5\%	\\
		S.e. $\widetilde{\theta}_n$	&	0.1981		&	0.4780	&	0.0063		&	S.e. $\hat{\theta}_n$		&	 1.2660		&	0.0539	&	0.0540	\\
		\hline
		$M_i^0 = 0.1$								&	$\lambda$	&	$K$			&	$\sigma$	&	$M_i^0 = 0.1$						&	$\lambda$	&	$K$			&	$\sigma$\\
		\hline
		True ($\theta_0$)						&	3					&	5				&	0.1				&	 True ($\theta_0$)				&	3					&	5				&	0.1			\\
		Mean $\widetilde{\theta}_n$	& 2.9950		&	4.9926	&	0.1036		&	Mean $\hat{\theta}_n$		&	 3.1261		&	4.8713	&	0.2425	\\
		Bias $\widetilde{\theta}_n$	& -0.2\%		&	-0.1\%	&	3.6\%			&	Bias $\hat{\theta}_n$		 &	4.2\%			&	-2.6\%	&	142.5\%	\\
		S.e. $\widetilde{\theta}_n$	&	0.2196		&	0.0708	&	0.0121		&	S.e. $\hat{\theta}_n$		&	 1.2823		&	0.0320	&	0.0569	\\
		\hline
		$M_i^0 = 0.1$								&	$\lambda$	&	$K$			&	$\sigma$	&	$M_i^0 = 0.1$						&	$\lambda$	&	$K$			&	$\sigma$\\
		\hline
		True ($\theta_0$)						&	3					&	5				&	0.5				&	 True ($\theta_0$)				&	3					&	5				&	0.5			\\
		Mean $\widetilde{\theta}_n$	& 2.9866		&	5.0513	&	0.4974		&	Mean $\hat{\theta}_n$		&	 2.7822		&	4.9437	&	0.5126	\\
		Bias $\widetilde{\theta}_n$	& -0.4\%		&	1.0\%		&	-0.5\%		&	Bias $\hat{\theta}_n$		 &	-7.3\%		&	-1.1\%	&	2.5\%		\\
		S.e. $\widetilde{\theta}_n$	&	0.2883		&	0.1505	&	0.0226		&	S.e. $\hat{\theta}_n$		&	 1.2616		&	0.1524	&	0.1288	\\
		\hline
		\end{tabular}
		\caption{Parameter (re)estimation of $\lambda_0$, $K_0$ and $\sigma_0$ using the non-stationary ML-estimator $\widetilde{\theta}_n$
						and the stationary ML-estimator $\hat{\theta}_n$.
						The estimates are based on 30 realisations of $\{\Phi_M(t)\}_{t\in[0,100]}$ (non-stationary) sampled at $T_k=k$, $k=1,\ldots,100$
						(with time discretisation step $dt=0.01$) which are generated from $\alpha_0=0.5$, $\mu_0=0.01$, $W=[0,1]^2$ and the above parameters ($M_i^0$ and $\theta_0$).
						}
	\label{TableEvaluation}
\end{table}

\section{Modelling Scots pines}\label{SectionPines}

As previously mentioned, the SG-process is constructed as a stochastic extension of the GI-process, under the assumption that the interaction between the marks is negligible. 
Hence, when considering the GI-process' main application area, which is the dynamical modelling of forest stands, it makes sense to employ the stationary mark SG-process when we want to model a homogenous forest stand (trees of the same species with similar ages) where e.g.\ the distances between the trees are large (we may ignore the interaction). 

One data set which (arguably) may be considered to fulfil these requirements is the set of Swedish Scots pines considered in \cite{Cronie}, which is illustrated in Figure \ref{PinesExample} (all tree radii have been scaled by a factor of 10 for increased visibility). 
The spatial region $W$ under consideration here is given by a circular region of radius 10 meters and the actual data set is given by a time series of marked point patterns, recorded at the years 1985, 1990 and 1996, where the approximate age of the forest stand in 1985 was 22 years. 
Hereby we may set $T_1 = 22$, $T_2 = 27$ and $T_3 = 33$, and we have $N(T_1)=13$, $N(T_2)=26$ and $N(T_3)=43$.
To be precise, for each $T_k$, $k=1,2,3$, each marked point pattern consists of measurements 
of radii (at breast height) $m_{ik}$ and locations (stock centres) $x_i\in W$ of the trees $i\in\{j:m_{jk}>0\}$ which are present at $T_k$, and only trees having reached a radius of 0.005 meter are included in the data set.

\begin{figure}[!htbp]
	\begin{center}
		\begin{tabular}{ccc}
			{\includegraphics[width=0.3\textwidth]{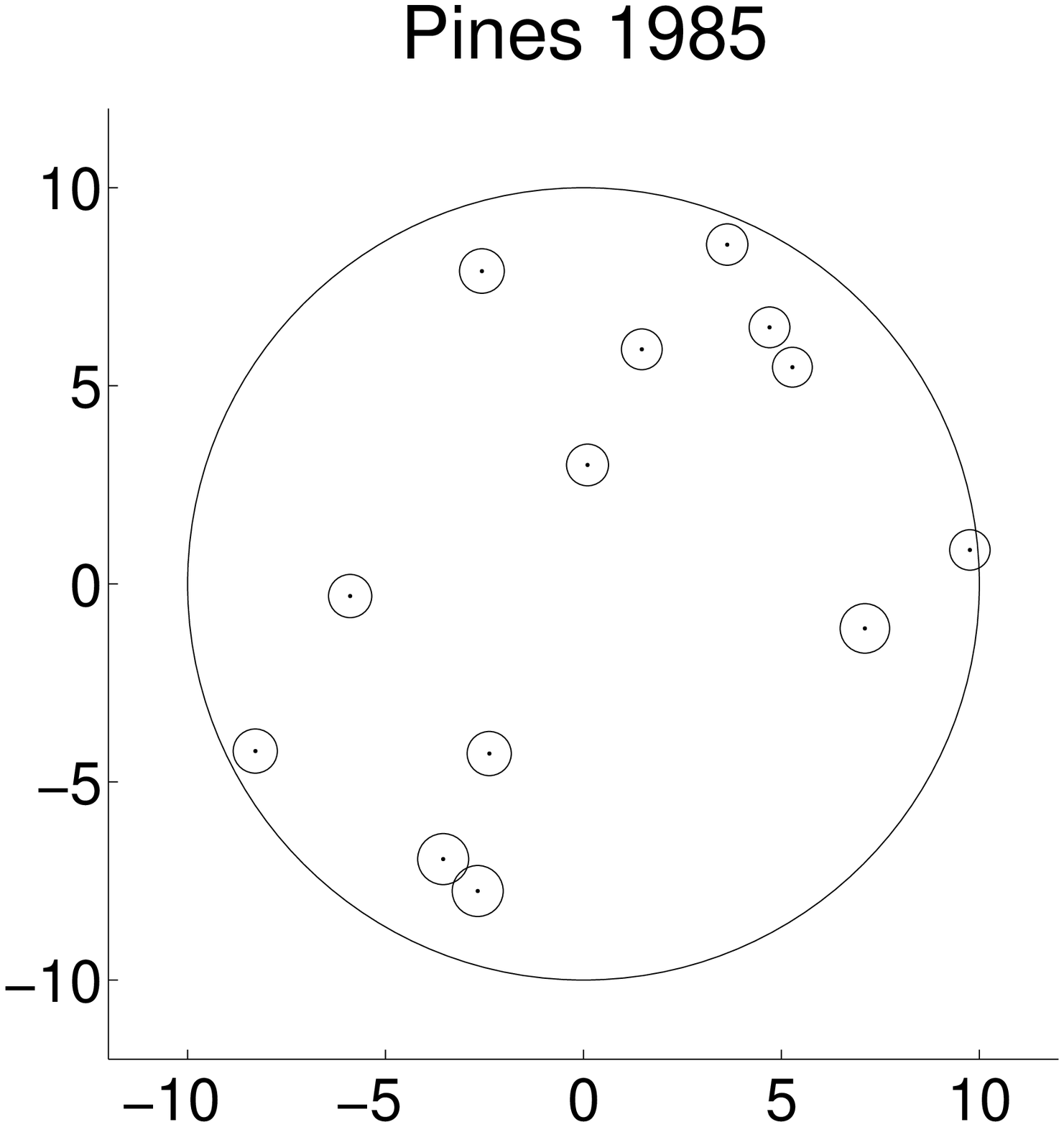}} &
			{\includegraphics[width=0.3\textwidth]{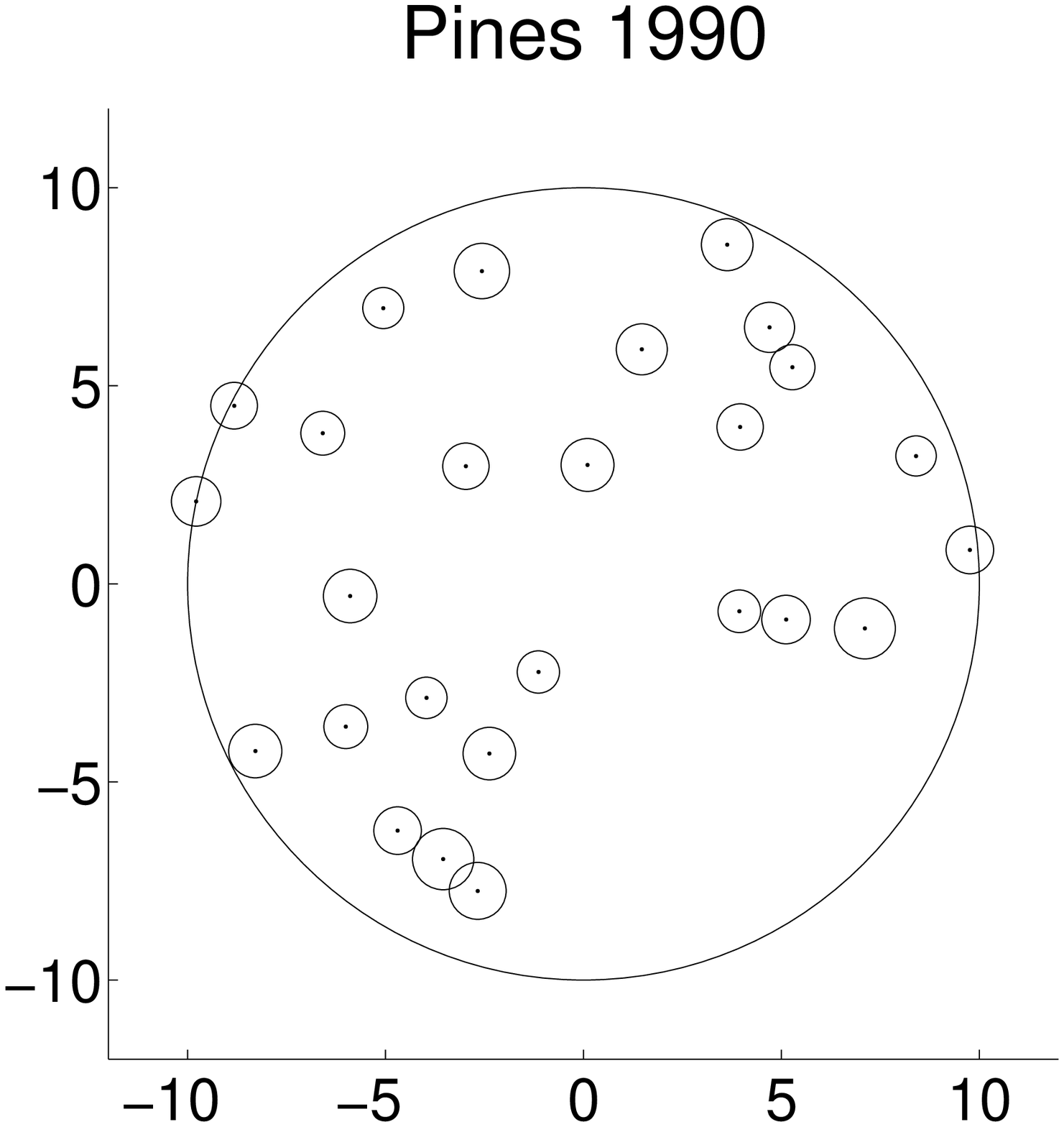}} &
     			{\includegraphics[width=0.3\textwidth]{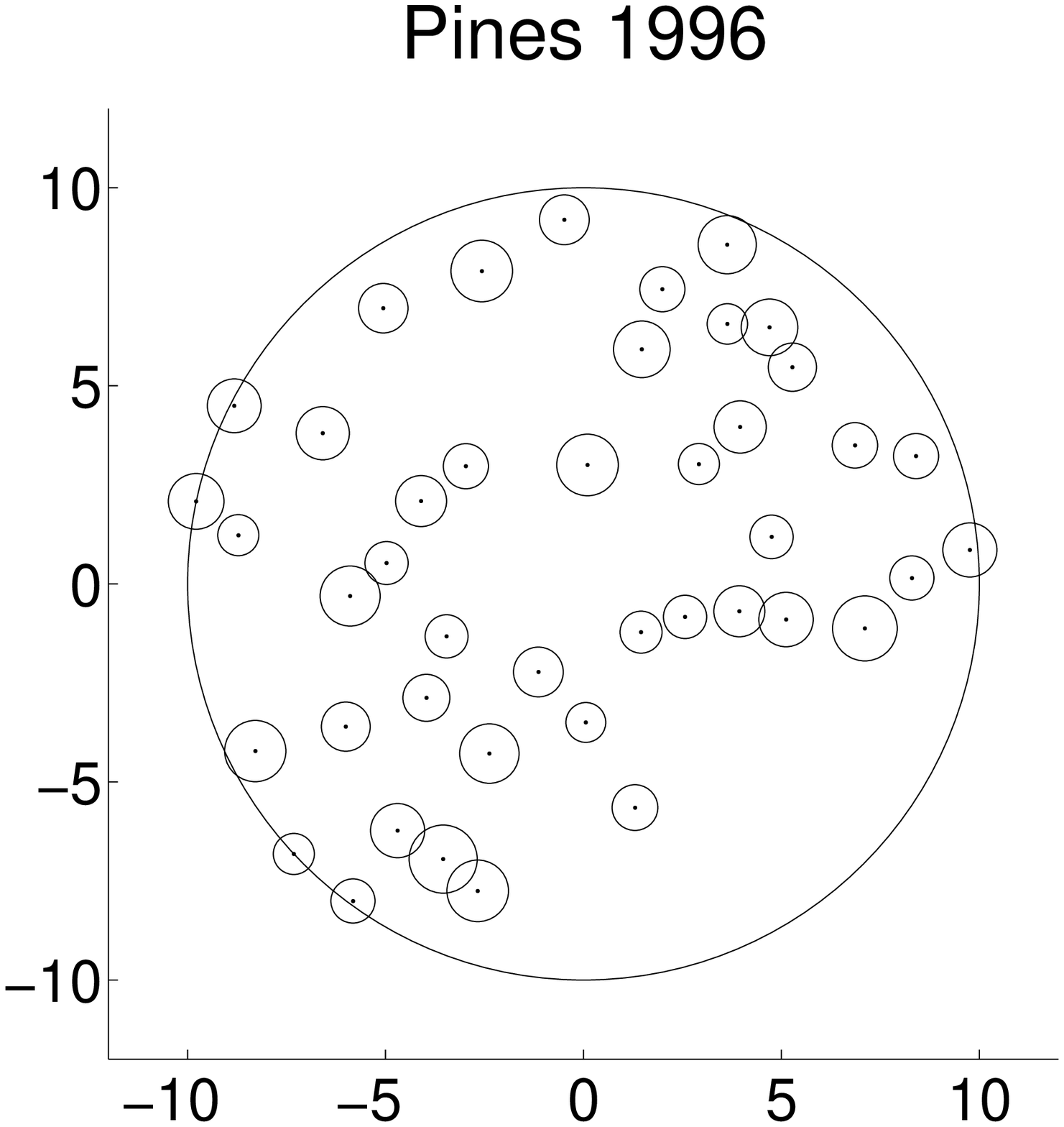}}
   	\end{tabular}
	\end{center}
	\caption{Swedish Scots pines: plots recorded in $1985$ (left), $1990$ (middle) and $1996$ (right). The radii of the pines are scaled by a factor of 10.}
	\label{PinesExample}
\end{figure}
The approach used in \cite{Cronie} to model this data set was to employ the so-called logistic growth function $f(Y_{i}(t);\theta) = \lambda Y_{i}(t)\left(1-Y_{i}(t)/K\right)$ as individual/open growth function in (\ref{GI}) 
and the so-called area-interaction function $h(\cdot)$ (see expression (\ref{GI})) to describe the spatial interaction between the marked points. 
We note that both this individual growth function and the (linear growth function) drift coefficient $f(Y_{i}(t);\theta) = \lambda\left(1-Y_{i}(t)/K\right)$ in the CIR process are special cases of the so-called Von Bertalanffy-Chapman-Richards (VBCR) growth function (see e.g.\ \cite{RenshawComasMateu}), whence their behaviours are quite similar. 

As previously mentioned, besides $\alpha$ and $\mu$, the parameters under consideration here are the growth rate $\lambda$, the carrying capacity $K$ and the diffusion parameter $\sigma$. 
In Table \ref{ResultsPines} we find, together with the results obtained in \cite{Cronie}, the results obtained after having fit the SG-process to the data set in Figure \ref{PinesExample}. 
Note that the choice $M_i^0=0.005$ has been made (in the non-stationary SG-process and in the GI-process) since the trees in the data set have been measured only once they have grown to at least a radius (at breast height) of 0.05 meter.
Regarding the estimation of $(\alpha,\mu)$, \cite{Cronie} obtained $\hat{\alpha} = 0.0042$ and $\hat{\mu} = 0$ (based on the estimators given therein). 
Here, we obtain $(\hat{\alpha},\hat{\mu})=(0.0613, 0.7020)$ whence, once the forest stand has become old, we would expect $\hat{\alpha}\nu(W)/\hat{\mu} \approx 27$ trees in $W$.
\begin{table}[htbp]
	\centering
	\begin{tabular}{lccc}
\hline
			&$\hat{\lambda}$	& $\hat{K}$	&$\hat{\sigma}$\\
\hline
GI			&	0.078		&	0.095	&	--		\\
SG			&	0.371		&	0.073	&	0.151	\\
Stationary SG	&	1.269		&	0.062	&	0.218	\\
		\hline
\end{tabular}
\caption{Results obtained after fitting the stationary mark SG-process and the GI-process (results from \cite{Cronie}) to a data set of Swedish Scots pines. }
	\label{ResultsPines}
\end{table}

It comes as no surprise that $\hat{K}$ is larger in the GI-process than in the SG-process. 
This follows since in the GI-process, the estimation of the open growth ($\lambda$ and $K$ in the logistic growth function) takes into account also that the observed sizes $m_{ik}$ are results of an open growth which has been inhibited by spatial interaction, i.e.\ $f(\cdot)$ is inhibited by $h(\cdot)$. 
Since $\max_{i,k}m_{ik}=0.0860$ (see \cite{Cronie}) it is probable that the SG-process underestimates $K$ a bit. 
Moreover, by comparing the results for the stationary and the non-stationary SG-process, we conclude that an increased $\hat{\sigma}$ (larger fluctuations) for the stationary case also results in a stronger estimated mean reversion (increased $\hat{\lambda}$).

From the differences in $\hat{\lambda}$ and $\hat{\sigma}$ for the two SG-processes 
we have indications that the data set has not (yet) reached stationarity, which is to be expected since the forest stand we are considering is quite young. 

In conclusion, mainly due to the difference in $\hat{K}$ between the SG-process and the GI-process as well as the sensibility of having stochastic marks in the GI-process, 
this pilot study certainly motivates a further investigation of the applicability of the full SGI-process, where we include an interaction function $h(\cdot)$ in the drift term of each diffusion $M_i(t)$, $t\in[B_i,D_i)$, i.e.\ where we add a stochastic integral term $\int_{B_i}^{D_i}\sigma(M_i(t))dW_i(t)$ to expression (\ref{GI}).

\section{Discussion}\label{SectionDiscussion}

We have here considered the GI-process only in the context of the CIR-mark process, but we may just as well employ any other positive diffusion for the growth of the marks.
As previously noted, the linear growth function, which is the drift function in the CIR-process, is a special case of the Von Bertalanffy-Chapman-Richards (VBCR) growth function (see e.g.\ \cite{RenshawComasMateu}). 
Another special case of the VBCR growth function is the 
aforementioned logistic growth function $f(Y_{i}(t);\theta) = \lambda Y_{i}(t)\left(1-Y_{i}(t)/K\right)$ which has been used in the GI-process in e.g.\ \cite{Cronie,RS1,RS2}.

A further modification which may be made is to change the diffusion term $\sigma(Y_{i}(t);\theta) = \sigma\sqrt{Y_{i}(t)}$ into any other diffusion term which keeps $Y_{i}(t)$ positive, e.g.\ $\sigma(Y_{i}(t);\theta) = \sigma Y_{i}(t)^{\gamma}$, $\gamma>0$, which is the diffusion coefficient found in the CKLS-model (see e.g.\ \cite{CKLS}).
Note that when applying these changes, we would typically not have known closed form expressions for the transition densities, $p_{Y_i}(t,y_1|y_0;\theta)$.
The transition densities are know only for a few special cases, including the CIR-process. Therefore, we have to use different approximated/pseudo likelihood methods for the estimation of the parameters (see \cite{Iacus} for a good general overview).

Our final goal is to ML-estimate all parameters of the full SGI-process, i.e.\ to include also the spatial interaction function $h(\cdot)$ in expression (\ref{GI}).
Here the lack of closed form expressions for the transition densities remains and, just as for the previous adjustments suggested, the estimation requires that we employ approximated/pseudo likelihood methods.
For instance, \cite{Ait} suggests an approach where the transition densities of multivariate diffusions may be approximated by series expansions based on hermite polynomials.
Note further that within this setting, in order to reduce edge effects (absence of individuals outside the boundary of $W$), it would be sensible to choose $W$ to be a torus.

Also, thus far we have introduced the type of death which occurs when $t\geq D_i$, i.e.\ the life-time of the individual has expired.
Following the terminology of \cite{RS2}, we can refer to this type of death as \emph{natural death}.
It is possible, however, to introduce another type of death, namely so called \emph{competitive death} (or \emph{interactive death}), and its introduction entails a slightly different formulation of the diffusions $M_{i}(t)$, $i=1,\ldots,N$.
By defining the \emph{death-time} of individual $i$ to be (the stopping-time) $\zeta_i = \inf\{t>B_i:M_{i}(t) = 0\}\wedge D_i$, we have that
if $M_i$ reaches the absorbing state $M_{i}(t)=0$ for some $t\in(B_i,D_i)$ it stays 0 and we say that it has suffered a competitive death.
Furthermore, if it does not die from competition in $(B_i,D_i)$ it will still die at time $D_i$, i.e.\ at its natural death time.
As soon as $t>\zeta_i$ the interaction between $M_{i}(t)$ and the other marks will terminate, hence we remove individual $i$ from consideration.

\section*{Acknowledgements}
The author would like to thank Peter Guttorp (University of Washington), Aila S\"arkk\"a (Chalmers university of technology) and Jun Yu (Swedish university of agricultural sciences) for useful suggestions and discussions.

\appendix
\section{Appendix: Proofs}

\subsection{Proof of Proposition \ref{fidi}}
The proof of Proposition \ref{fidi} exploits the Markov property of $\Phi_M(t)$.

\begin{proof}[Proof of Proposition \ref{fidi}]
We first note that, by construction, when evaluated at $\mathbf{M}$, we may express the the joint density of $(\Phi_M(T_1),\ldots,\Phi_M(T_n))^T$ through its two building blocks -- the CIR-process and the underlying process $\Phi(t)$.
More specifically, we have that
\beann
\mathbf{p}_{T_1,\ldots,T_n}(\mathbf{M};\theta) &=& \mathbf{p}_{M(T_1,\ldots,T_n)|\Phi}(\mathbf{M};\lambda,K,\sigma) p_{\Phi}(\fomega,\{x_i\}_{i=1}^{d};\alpha,\mu),
\eeann
where $p_{\Phi}(\fomega,\{x_i\}_{i=1}^{d};\alpha,\mu)$ is the density of $(\Phi(T_1),\ldots,\Phi(T_n))^T$, evaluated at the index sets $(\Omega(T_1),\ldots,\Omega(T_n))^T = \fomega = (\omega_1,\ldots,\omega_n)^T$, $\omega_k = \{i:m_{ik} > 0\}$, and the locations $(X_i)_{i=1}^{d} = (x_i)_{i=1}^{d}$.
We note that the total number of individuals under consideration hereby is given by $d = \left|\bigcup_{k=1}^{n}\omega_k\right|$.
Furthermore, the density $\mathbf{p}_{M(T_1,\ldots,T_n)|\Phi}(\mathbf{M};\lambda,K,\sigma)$ is the conditional density of the diffusions, given $(\Omega(T_1),\ldots,\Omega(T_n))^T = \fomega$, and we note that, probabilistically, this is a statement only about the diffusions $Y_1(t),\ldots,Y_d(t)$.

We start by considering the part concerning the underlying process' behaviour at $T_1,\ldots,T_n$.
The density $p_{\Phi}(\fomega,\{x_i\}_{i=1}^{d};\alpha,\mu)$ can be further rewritten as the product
\beann
p_{\Phi}(\fomega,\{x_i\}_{i=1}^{d};\alpha,\mu)
= p_{\Phi|X}(\fomega;\alpha,\mu) \times p_X(\{x_i\}_{i=1}^{d};\nu(W)),
\eeann
where $p_{\Phi|X}(\fomega;\alpha,\mu)$ is the conditional density of $(\Omega(T_1),\ldots,\Omega(T_n))^T$, given $(X_i)_{i=1}^{d} = (x_i)_{i=1}^{d}$, and $p_X(\{x_i\}_{i=1}^{d};\nu(W))$ is the density of the locations $(X_i)_{i=1}^{d}$.
We thus conclude that
$p_{\Phi|X}(\fomega;\alpha,\mu)$ is a statement only about which intervals $[B_i,D_i)$ that cover $T_1,\ldots,T_n$.
Now, by letting $d_1 = |\omega_1|,\ldots,d_n = |\omega_n|$ and recalling the ID-process $N(t) = |\Omega(t)|$, when additionally conditioning $p_{\Phi|X}(\fomega;\alpha,\mu)$ on $(N(T_k))_{k=1}^{n} = (d_k)_{k=1}^{n}$, we obtain
\bea
\label{JointMarkedImmigrationDeath}
p_{\Phi}(\fomega,\{x_i\}_{i=1}^{d};\alpha,\mu)	
&=&	p_{\Phi|X}(\fomega;\alpha,\mu) p_X(\{x_i\}_{i=1}^{d};\nu(W)) \\
&=&	p_{\Phi|X,N}((\omega_k)_{k=1}^{n})\; p_{N}((d_k)_{k=1}^{n};\alpha,\mu)\nn\\
&&\times p_X(\{x_i\}_{i=1}^{d};\nu(W)),\nn
\eea
where $p_{N}((d_k)_{k=1}^{n};\alpha,\mu)$ is the density of $(N(T_k))_{k=1}^{n}$, evaluated at $(d_k)_{k=1}^{n}$, and $p_{\Phi|X,N}((\omega_k)_{k=1}^{n})$ is a statement about the order of appearance of the individuals in the index sets.

Starting with the last of the components of expression (\ref{JointMarkedImmigrationDeath}),
we clearly see that $p_X(\{x_i\}_{i=1}^{d};\nu(W)) = \nu(W)^{-d}$, since the $X_i$'s are independent and uniformly distributed over $W$.
Moreover, from the Markov property of $N(t)$ we have that $p_{N}((d_k)_{k=1}^{n};\alpha,\mu)$ may be written as a product of its transition densities, i.e.\
\beann
p_{N}((d_k)_{k=1}^{n};\alpha,\mu) = \prod_{k=1}^{n} p_{N}\left(\Delta T_{k},d_k \Big| d_{k-1};\alpha\nu(W),\mu\right),
\eeann
where $\Delta T_{k} = T_{k}-T_{k-1}$ and $p_{N}(t,y|x;\alpha\nu(W),\mu)$ is given by Lemma \ref{ImDeTransProb}.
Regarding the first of the parts constituting expression (\ref{JointMarkedImmigrationDeath}), it is given by
$p_{\Phi|X,N}((\omega_k)_{k=1}^{n}) = 1/|\mathcal{P}|$,
where
\beann
\mathcal{P} =
\Bigg\{
\fomega &:&
\omega_1,\ldots,\omega_n\subseteq\{1,\ldots,d\} = \bigcup_{k=1}^{n}\omega_k
;\ |\omega_1| = d_1,\ldots,|\omega_n|=d_n; \\
&&i\in\omega_{k-1}, i\notin\omega_{k} \Rightarrow i\notin\omega_{k+1} \text{ for any }k=2,\dots,n-1
\Bigg\}.
\eeann
This can be seen by considering the matrix $A = [a_{i,k}]_{i=1,\ldots,d;\ k=1,\ldots,n}$, which has entries $a_{i,k} = \1\{T_k\in[B_i,D_i)\}$.
When we condition on $(N(T_k))_{k=1}^{n} = (d_k)_{k=1}^{n}$ we only specify that the column-sums of $A$ are given by $d_1,\ldots,d_n$, whence we still have to determine what the probability is of $A$ being observed as the matrix $\bar{A}$ with entries $\bar{a}_{i,k} = \1\{m_{ik}>0\}$.
It may be seen that the sample space of the conditional random matrix $A|\{(N(T_k))_{k=1}^{n} = (d_k)_{k=1}^{n}\}$ is given by the $d\times n$-matrices which have 0-1 entries, column-sums $d_1,\ldots,d_n$ and all 1's in each row connected (this follows since individuals cannot start living again once they have died); when $n=5$, say, a row can be given by e.g.\ $(0,1,1,0,0)$ or $(1,1,1,1,0)$.
To obtain a bound for $|\mathcal{P}|$ we find that the number of matrices which have rows with connected 1's is given by $d\sum_{j=1}^{n}(n-(j-1)) = dn(n+1)/2$, whereby $2/dn(n+1)\leq1/|\mathcal{P}|\leq1$.

Hence, we may summarise expression (\ref{JointMarkedImmigrationDeath}) as
\beann
p_{\Phi}(\fomega,\{x_i\}_{i=1}^{d};\alpha,\mu) &=& C \prod_{k=1}^{n} p_{N}\left(\Delta T_{k},d_k \Big| d_{k-1};\alpha\nu(W),\mu\right),
\eeann
where the constant
\bea
\label{NormalizingConstant}
C = \nu(W)^{-d} |\mathcal{P}|^{-1}
\eea
depends on $\nu(W)$ and $\mathbf{M}$ (in particular $d$, $d_1,\ldots,d_n$).

We now turn to the part of the density which is related to the mark processes.
Due to the independence of the marks, the Markovianity of the CIR-process and the uniformly distributed birth times, we have that
\begin{align}
\label{MarkJointDensity}
&\mathbf{p}_{M(T_1,\ldots,T_n)|\Phi}(\mathbf{M};\lambda,K,\sigma)
=	\prod_{k=1}^{n} \prod_{i\in\omega_{k-1}\cap\omega_{k}} p_{Y_{i}}(\Delta T_{k}, m_{ik}|m_{i(k-1)};\lambda,K,\sigma)\nn\\
&\times \prod_{i\in\bigcup_{k=1}^{n}\omega_k} \int_{T_{k_{i}-1}}^{T_{k_{i}}} \frac{p_{Y_{1}}(T_{k_{i}}-t, m_{i(k_{i}-1)}|M_0;\lambda,K,\sigma)}{T_{k_{i}}-T_{k_{i}-1}} dt,
\end{align}
where $p_{Y_{i}}(t-s, y_{t}|y_{s};\lambda,K,\sigma)$ is given by expression (\ref{TransitionCIR}) and $k_{i}=\min\{k:i\in\omega_k\}$.
The first part of the above expression includes all considered transition time pairs $(T_{k-1},T_{k})$ whereas the second part includes the transitions between the (unobserved) arrival time $B_i\in(T_{k_i-1},T_{k_i})$ and the first sample time $T_{k_i}$ at which the individual is observed.

Although there is no information available regarding the exact death times and death sizes, it could be argued that the (unobserved) death sizes have been ignored in expression (\ref{MarkJointDensity}).
Letting $\bar{k}_{i}=\max\{k:i\in\omega_k\}$ be the index of the last sample time point at which the individual was alive and $T_{\bar{k}_{i}+1} = \infty$ if $\bar{k}_{i}=n$, we obtain
\beann
\int_{T_{\bar{k}_{i}}}^{T_{\bar{k}_{i}+1}} \int_{0}^{\infty}
\frac{p_{Y_{1}}(t - T_{\bar{k}_{i}}, x|m_{i\bar{k}_{i}};\lambda,K,\sigma)}{T_{\bar{k}_{i}+1} - T_{\bar{k}_{i}}} dx\,dt
=
\int_{T_{\bar{k}_{i}}}^{T_{\bar{k}_{i}+1}}
\frac{1}{T_{\bar{k}_{i}+1} - T_{\bar{k}_{i}}} dt = 1
\eeann
as (possible) contribution to expression (\ref{MarkJointDensity}).
Hence, the death sizes may be neglected.

\end{proof}

\subsection{Proofs of Theorem \ref{ThmCons} and Theorem \ref{ThmAsNorm}}

The consistency is shown by using the classical Wald approach and the asymptotic normality proof follows the approach of Cram\'er.

Before turning to the proofs, we first note some (well known) results used in the proofs of the consistency and the asymptotic normality of the sequence of ML-estimators $\hat{\theta}_n$.

The following lemma, which can be found in \cite{Ferguson}, will be used in both the consistency proof and the proof of the asymptotic normality.
\begin{lem}[Uniform Strong Law of Large Numbers]\label{USLLN}
Given that $Z_1,Z_2,\ldots$ are iid copies of the random variable $Z$, assume that:
\begin{itemize}
\item[(i)] $\Theta$ is compact,
\item[(ii)] $U(x,\theta)$ is upper semi-continuous in $\theta$ for all $x$ and there exists a function $K(x)$ such that $\E[K(X)]<\infty$ and $U(x,\theta)\leq K(x)$ for all $x$ and $\theta$,
\item[(iii)] for all $\theta$ and for all sufficiently small $\rho>0$, $\sup_{\theta:d(\theta',\theta)<\rho} U(x,\theta')$ is measurable in $x$.
\end{itemize}
Then
\beann
\P\left(
\limsup_{n\rightarrow\infty} \sup_{\theta\in\Theta} \frac{1}{n}\sum_{j=1}^{n} U(Z_j,\theta) \leq \sup_{\theta\in\Theta}\mu(\theta)
\right)=1,
\eeann
where $\mu(\theta)=\E[U(Z,\theta)]$.
If we replace (ii) and (iii) by
\begin{itemize}
\item[(ii)'] $U(x,\theta)$ is continuous in $\theta$ for all $x$ and there exists a function $K(x)$ such that $\E[K(X)]<\infty$ and $|U(x,\theta)|\leq K(x)$ for all $x$ and $\theta$,
\end{itemize}
we obtain instead
\beann
\P\left(
\lim_{n\rightarrow\infty} \sup_{\theta\in\Theta} \left|\frac{1}{n}\sum_{j=1}^{n} U(Z_j,\theta) - \mu(\theta)\right| = 0\right) = 1.
\eeann
\end{lem}

A further convergence lemma, Slutsky's lemma, which can be found in e.g.\ \cite{Ferguson}, is used both in the consistency proof and in the proof of the asymptotic normality.
It combines converging stochastic sequences.

The Lindeberg-Feller central limit theorem (see e.g.\ \cite{VanDerVaart}), which we will exploit in the proof of Theorem \ref{ThmAsNorm}, gives us the asymptotic normality of sums of independent random vectors which are not necessarily identically distributed.

We here (partly) will consider a stronger version of the Lindeberg-Feller central limit theorem, which is given by a multivariate form of the Lyapunov central limit theorem, which can be found in e.g.\ \cite{Hoadley} (it is stronger in the sense that if the Lyapunov condition holds then the Lindeberg condition is satisfied (see e.g.\ \cite{Gut})).

\begin{proof}[Proof of Theorem \ref{ThmCons}]
From \cite{CronieYu} we already have that the ML-estimator of the ID-process is strongly consistent, i.e.\
$(\hat{\alpha}_n,\hat{\mu}_n) \stackrel{a.s.}{\longrightarrow} (\alpha_0,\mu_0)$, as $n\rightarrow\infty$.
Hence, if we manage to show that $(\hat{\lambda}_n,\hat{K}_n,\hat{\sigma}_n) \stackrel{a.s.}{\longrightarrow} (\lambda_0,K_0,\sigma_0)$, as $n\rightarrow\infty$, then Slutsky's lemma (see e.g.\ \cite{Ferguson}) gives us that $\hat{\theta}_n \stackrel{a.s.}{\longrightarrow} \theta_0$.

To simplify the notation we write $\Theta$ for $\Theta_{\lambda}\times\Theta_{K}\times\Theta_{\sigma}$ so that $\theta = (\lambda,K,\sigma)\in\Theta$, and it now remains to show that $\tilde{\theta}_{n}=(\hat{\lambda}_{n},\hat{K}_{n},\hat{\sigma}_{n}) \stackrel{a.s.}{\longrightarrow} (\lambda_0,K_0,\sigma_0) = \theta_0$, as $n\rightarrow\infty$.
The idea of the proof is to show that, for any $\delta>0$, if we assume that $\tilde{\theta}_n\in \{\theta\in\Theta:d(\theta,\theta_0)\geq\delta\}$ we get a contradiction.

By suppressing all conditioning in the iid random variables $M_{i}(T_k)|M_{i}(T_{k-1})$ and relabeling the $A(n)=\sum_{k=1}^{n}|\Omega(k\Delta)|$ observations of the $M_i(T_k)$'s up to time $T_n$ as $Z_1,\ldots,Z_{A(n)}$, we may write the log likelihood as
\bea
\label{LogLikRelabled}
l_{1,n}(\theta)	&=& \sum_{i=1}^{A(n)} l_{i}(\theta),\\
l_{i}(\theta)
&=&\log\pi(Z_{i};\theta)
= \log\left\{\frac{Z_{i}^{2\lambda/\sigma^2-1} \e^{-Z_{i}(2\lambda/\sigma^2 K)}}{\left(\sigma^2K/2\lambda\right)^{2\lambda/\sigma^2}\Gamma(2\lambda/\sigma^2)}\right\}\nn\\
&=& (2\lambda/\sigma^2-1)\log Z_{i} - Z_{i}(2\lambda/\sigma^2 K) - \log\Gamma(2\lambda/\sigma^2)\nn\\
&&-(2\lambda/\sigma^2)\left(2\log\{\sigma\} + \log\{K\} - \log\{2\} - \log\{\lambda\}\right).\nn
\eea
We note further that $A(n)$ is non-decreasing in $n$.
Since the total number of individuals is given by $N\sim Poi(\alpha n\Delta)$, we have that $N\stackrel{a.s.}{\longrightarrow} \infty$, as $n\rightarrow\infty$, and by the strong law of large numbers we have that $\frac{1}{N}\sum_{i=1}^{N}\1\{L_i>\Delta\} \stackrel{a.s.}{\longrightarrow} \E[\1\{L_1>\Delta\}] = \e^{-\Delta\mu}>0$, as $N\rightarrow\infty$, for any $(\alpha,\mu)\in\Theta_{\alpha}\times\Theta_{\mu}$.
Hence, as $n\rightarrow\infty$, we will observe the stationary diffusions $Y_i(t)$ an infinite number of times at our sampling times $T_1,\ldots,T_n$, whence $A(n) \stackrel{a.s.}{\longrightarrow} \infty$.
Note that if we for instance choose to include only one observation of each diffusion in (\ref{LogLikRelabled}) (say the last one), i.e.\ $A(n)=\left|\bigcup_{k=1}^{n}\Omega(k\Delta)\right|$, the convergence would still hold.

Treated as a function of $\theta$, we note that the MLE $\tilde{\theta}_n$ also maximises $\frac{1}{A(n)}l_{1,n}(\theta) - \frac{1}{A(n)}l_{1,n}(\theta_0) = \frac{1}{A(n)}\sum_{i=1}^{A(n)} (l_{i}(\theta)-l_{i}(\theta_0))$.
Consider now the function
\bea
\label{contraction}
U(x,\theta)	&:=& \log\pi(x;\theta) - \log\pi(x;\theta_0)\\
						&=& 2\left(\frac{\lambda}{\sigma^2} - \frac{\lambda_0}{\sigma_0^2}\right)\log(x)
						- 2\left(\frac{\lambda}{\sigma^2 K} - \frac{\lambda_0}{\sigma_0^2 K_0}\right)x
						+ \log\frac{\Gamma\left(\frac{2\lambda_0}{\sigma_0^2}\right)}{\Gamma\left(\frac{2\lambda}{\sigma^2}\right)}\nn\\
						&&+\;(2\lambda_0/\sigma_0^2)\left(2\log\{\sigma_0\} + \log\{K_0\} - \log\{2\} - \log\{\lambda_0\}\right)\nn\\
						&&-\;(2\lambda/\sigma^2)\left(2\log\{\sigma\} + \log\{K\} - \log\{2\} - \log\{\lambda\}\right),\nn
\eea
which clearly is continuous in both arguments, and thereby a measurable function of $x$.

\underline{Identifiability}: The only way in which expression (\ref{contraction}) can be set equal to $0$ is to require that $\theta=\theta_0$.
Hence, $\{\pi(\cdot;\theta):\theta\in\Theta\}$ is an identifiable family of distributions, and this in turn guarantees that the ML-estimator converges to the unique maximum $\theta_0$ (the uniqueness follows from e.g.\ Lemma 5.35, p. 62, in \cite{VanDerVaart}).

\underline{Conditions of Lemma \ref{USLLN}}:
First we note that the continuity of $U(x,\theta)$ implies that it is upper semi-continuous in $\theta$ for all $x$.
Now, in order to find the required bound $K(x)$, we note that
\begin{align*}
&U(x,\theta)
\leq 2\sup_{\theta\in\Theta}|\log\pi(x;\theta)|\\
&\leq 2\left(\frac{2\bar{\lambda}}{\underline{\sigma}^2} + 1\right)|\log(x)| + \frac{4\bar{\lambda}}{\underline{\sigma}^2\underline{K}}x
+ 2\log\left\lceil2\bar{\lambda}/\underline{\sigma}^2\right\rceil!\\
&+\frac{4\bar{\lambda}}{\underline{\sigma}^2} \left(\sup_{\lambda\in\theta_{\lambda}}|\log\{\lambda\}| + \sup_{K\in\theta_{K}}|\log\{K\}| + 2\sup_{\sigma\in\theta_{\sigma}}|\log\{\sigma\}| + \log\{2\}\right)\\
&=:K(x),
\end{align*}
where e.g.\ $\bar{\lambda}=\sup\{\Theta_{\lambda}\}$, $\underline{\sigma}=\inf\{\Theta_{\sigma}\}$ and $\left\lceil x \right\rceil = \inf\{y\in\N:x\leq y\}$ is the ceiling function.
Because of the boundedness (compactness) of $\Theta$ we have that $\E_{\theta_0}[|K(Z_i)|]<\infty$ since also $\E_{\theta_0}[Z_i]=K_0$ and (by Jensen's inequality and Appendix \ref{SectionAppendixDerivatives})
\beann
\E_{\theta_0}[|\log(Z_i)|] &=& \E_{\theta_0}[\sqrt{\log(Z_i)^2}]
\leq \sqrt{\E_{\theta_0}[\log(Z_i)^2]}\\
&=& \Bigg(
\log\left(\frac{2\lambda_0}{K_0\sigma_0^2}\right)^2 - 2\log\left(\frac{2\lambda_0}{K_0\sigma_0^2}\right)\psi\left(\frac{2\lambda_0}{\sigma_0^2}\right)\\
&&+ \psi\left(\frac{2\lambda_0}{\sigma_0^2}\right)^2 + \psi'\left(\frac{2\lambda_0}{\sigma_0^2}\right)
\Bigg)^{1/2} < \infty,
\eeann
where $\psi(x)=\Gamma'(x)/\Gamma(x)$.

By the measurability of $U(x,\theta)$, we now may define the continuous function $\mu(\theta)$, which by \cite{Mathiassen} is given by
\beann
\mu(\theta) &=& \E_{\theta_0}[U(Z_i,\theta)]
						= -D(\pi_{\theta_0}||\pi_{\theta})\\
					&=& \psi\left(\frac{2\lambda_0}{\sigma_0^2}\right)\left(\frac{2\lambda_0}{\sigma_0^2} - \frac{2\lambda}{\sigma^2}\right)
					- \frac{2\lambda_0}{\sigma_0^2}
					+ \log\frac{\Gamma(2\lambda/\sigma^2)}{\Gamma(2\lambda_0/\sigma_0^2)}\\
					&&+ \frac{2\lambda}{\sigma^2}\log\left(\frac{\sigma^2 K \lambda_0}{\sigma_0^2 K_0 \lambda}\right)
					+ \frac{2\lambda K_0}{\sigma^2 K},
\eeann
where $D(\pi_{\theta_0}||\pi_{\theta})$ is the Kullback-Leibler divergence between the the two distributions involved.

From the Shannon-Kolmogorov inequality we have that $\mu(\theta) = -D(\pi_{\theta_0}||\pi_{\theta})\leq0$, with $\mu(\theta) = -D(\pi_{\theta_0}||\pi_{\theta})=0$ iff $\theta=\theta_0$.
Hence, by letting $\delta>0$ and defining the compact set $C = \{\theta\in\Theta:d(\theta,\theta_0)\geq\delta\}$, we note that by its continuity, $\mu(\theta)$ attains its (negative) maximum on $C$, i.e.\ $\bar{\mu} = \sup_{\theta\in C} \mu(\theta)<0$.

Since suprema of measurable functions are measurable we get that, for any $S\subseteq\Theta$, the function $\sup_{\theta\in S} U(x;\theta)$ is measurable in $x$.

\underline{Argument:}
Now, by Lemma \ref{USLLN} we have that
\beann
\P\left(\limsup_{n\rightarrow\infty} \sup_{\theta\in C} \frac{1}{A(n)}\sum_{j=1}^{A(n)} U(Z_j,\theta) \leq \bar{\mu}\right)=1,
\eeann
so that there a.s.\ exists an $n^*\in\N$ such that for all $n>n^*$,
\beann
\sup_{\theta\in C} \frac{1}{A(n)}\sum_{j=1}^{A(n)} U(Z_j,\theta) \leq \bar{\mu}/2 < 0,
\eeann
say. But at the same time, since $\frac{1}{A(n)}\sum_{j=1}^{A(n)} U(Z_j,\theta_0)=0$, we must have that
\beann
\frac{1}{A(n)}\sum_{j=1}^{A(n)} U(Z_j,\tilde{\theta}_n) = \sup_{\theta\in\Theta} \frac{1}{A(n)}\sum_{j=1}^{A(n)} U(Z_j,\theta) \geq 0.
\eeann
Hence, for $n>n^*$, this implies that $\tilde{\theta}_n\notin C$ a.s., or equivalently that $d(\tilde{\theta}_n,\theta_0)<\delta$.
Since $\delta>0$ was arbitrarily chosen we get that $\tilde{\theta}_n \stackrel{a.s.}{\longrightarrow} \theta_0$ as $n\rightarrow\infty$.

Note that we always can find a measurable selection $\hat{\theta}(\x)$ such that
$l_n(\x;\hat{\theta}(\x)) = \sup_{\theta\in\Theta}l_n(\x;\theta)$, for all $\x$ (see e.g.\ \cite{Ferguson}), whence the measurability of $\hat{\theta}_n$ never is addressed.

\end{proof}

\begin{proof}[Proof of Theorem \ref{ThmAsNorm}]
Given $\hat{\theta}_{n}^* = (\hat{\lambda}_{n},\hat{K}_{n},\hat{\sigma}_{n},\hat{\alpha}_{n},\hat{\mu}_{n})^T$ and $\theta_0^* = (\lambda_0,K_0,\sigma_0,\alpha_0,\mu_0)^T$, we are here interested in the asymptotic distribution of $\sqrt{n}(\hat{\theta}_{n}^* - \theta_0^*)$. 
Since either $\lambda_0$ or $\sigma_0$ is known, we will be dealing with the asymptotic distributions of
\begin{itemize}
\item[(i)]
$
\sqrt{n}(\hat{\theta}_{n}-\theta_0) = \sqrt{n}\left(\hat{K}_{n}-K_0,\hat{\sigma}_{n}-\sigma_0,\hat{\alpha}_{n}-\alpha_0,\hat{\mu}_{n}-\mu_0\right)^T,
$
\item[(ii)]
$
\sqrt{n}(\hat{\theta}_{n}-\theta_0) = \sqrt{n}\left(\hat{\lambda}_{n}-\lambda_0,\hat{K}_{n}-K_0,\hat{\alpha}_{n}-\alpha_0,\hat{\mu}_{n}-\mu_0\right)^T.
$
\end{itemize}
Unless necessary, we will not distinguish in the notation between the two scenarios above, and in what follows we will prove that, when $n\rightarrow\infty$, the random vector $\sqrt{n}(\hat{\theta}_{n}-\theta_0)$ will asymptotically have a Gaussian distribution.

From the consistency proof we recall $Z_i = M_{i}(T_k)|M_{i}(T_{k-1})$, $N_k=|\Omega(k\Delta)|$, $A(n)=\sum_{k=1}^{n}N_k$ and the (parameter reduced) log-likelihood function
\beann
l_{n}(\theta)	&=& l_{n}\left(\theta;\Phi_{M}(\Delta),\ldots,\Phi_{M}(n\Delta)\right)\\
							&=& l_{1,n}\left(\theta;(Z_{i})_{i=1}^{A(n)}\right) + l_{2,n}\left(\theta;(N_{i})_{i=1}^{n}\right)\\
							&=& l_{1,n}(\theta) + l_{2,n}(\theta)
							= \sum_{k=1}^{n} l_{\Y}\left(\theta;\Z_{k}\right) + \sum_{k=1}^{n} l_{N}(\theta;N_k,N_{k-1}),
\eeann
where $\Z_{k} = (Z_{i})_{i\in\Omega(k\Delta)}$ and
\beann
l_{\Y}(\theta;\Z_{k})	&=& \sum_{i\in\Omega(k\Delta)} l_{Y}(\theta;Z_{i}),\\
l_{Y}(\theta;Z_{i})
&=&
\left\{
\begin{array}{ll}
l_{Y}(K,\sigma;Z_{i})	= \log\pi(Z_{i};\lambda_0,K,\sigma)		& \text{if }\lambda_0\text{ is known}\\
l_{Y}(\lambda,K;Z_{i}) = \log\pi(Z_{i};\lambda,K,\sigma_0)	& \text{if }\sigma_0\text{ is known},
\end{array}
\right.\\
l_{N}(\theta;N_k,N_{k-1}) &=& \log p_{N}\left(\Delta,N_k \big| N_{k-1};\alpha,\mu\right).
\eeann
The indexation of the $Z_i$'s will have two meanings, which will be clear from the notation:
We either deal with $(Z_i)_{i\in\Omega(k\Delta)}$, $k=1,\ldots,n$, or $(Z_1,\ldots,Z_{A(n)})$.
Note that although we simply write $l_{Y}(\theta;Z_{i})$, there is still a dependence of $\{\Omega(k\Delta)\}_{k=1}^{n}$ present (recall the construction of $\L(\theta)$).

We will here denote by $S_{\rho} = \{\theta\in\Theta:d(\theta,\theta_0)\leq\rho\}$ the closed neighbourhood of $\theta_0$ with radius $\rho>0$ and we note that the consistency holds also for this (restricted) compact parameter space.
Now, given the conditions under which $\hat{\theta}_n$ was proved strongly consistent, we get that $\hat{\theta}_{n}$ is a strongly consistent sequence of roots of the likelihood equation $\dot{l}_{n}(\theta) = \frac{\partial l_{n}(\theta)}{\partial\theta} = \mathbf{0}_{4\times 1}$, where $\mathbf{0}_{i\times j}$ denotes the $i\times j$ zero matrix, i.e.\ $\dot{l}_{n}(\hat{\theta}_n)=\mathbf{0}_{4\times 1}$ (see e.g.\ thm 18, p. 121, \cite{Ferguson}).

As we shall see, the vector $\dot{l}_{n}(\theta)$ and the $4\times4$ matrix $\ddot{l}_{n}(\theta) = \partial^2 l_{n}(\theta)/\partial\theta^2$ are well behaved enough to Taylor expand $\dot{l}_{n}(\theta)$ around $\theta_0$:
\beann
\dot{l}_{n}(\theta) &=& \dot{l}_{n}(\theta_0) + \int_{0}^{1} \ddot{l}_{n}(\theta_0 + x(\theta-\theta_0))dx\,(\theta-\theta_0),\\
\text{or}\nn\\
\sqrt{n}(\theta-\theta_0)
&=& -\Bigg(\frac{1}{n} \int_{0}^{1} \ddot{l}_{n}(\theta_0 + x(\theta-\theta_0))dx\Bigg)^{-1} \frac{1}{\sqrt{n}} (\dot{l}_{n}(\theta_0) - \dot{l}_{n}(\theta))\\
&=& -B_{n}(\theta)^{-1} \frac{1}{\sqrt{n}} (\dot{l}_{n}(\theta_0) - \dot{l}_{n}(\theta)).
\eeann
We wish to prove that, when evaluated at $\theta=\hat{\theta}_{n}$, the right hand side of the last row of the above expression converges in law to a zero mean multivariate normal distribution.
By managing to show that $-B_{n}(\hat{\theta}_{n})\stackrel{a.s.}{\longrightarrow} I(\theta_0)$, then eventually $-B_{n}^{-1}(\hat{\theta}_{n})$ will exist, given that the inverse of the (asymptotic) Fisher information $I(\theta_0)$ exists.
Then, once we have shown that $\dot{l}_{n}(\theta_0)/\sqrt{n} \stackrel{d}{\longrightarrow} \mathbf{G} \sim N\left(\mathbf{0},I(\theta_0)\right)$ as $n\rightarrow\infty$, by means of Slutsky's lemma (see e.g.\ \cite{Ferguson}) we may establish that
\bea
\label{Taylor}
\sqrt{n}(\hat{\theta}_{n} - \theta_0) &=& -B_{n}^{-1}(\hat{\theta}_{n}) \frac{1}{\sqrt{n}}\dot{l}_{n}(\theta_0) \\
&\stackrel{d}{\longrightarrow}& I(\theta_0)^{-1}\mathbf{G}\sim N\left(\mathbf{0},I(\theta_0)^{-1}\right).\nn
\eea

The $4\times1$ vector $\dot{l}_{n}(\theta)$ of first order partial derivatives found in expression (\ref{Taylor}) is given by
\beann
\dot{l}_{n}(\theta)			&=& \dot{l}_{1,n}(\theta) + \dot{l}_{2,n}(\theta)
													= \sum_{k=1}^{n} \dot{l}_{\Y}\left(\theta;\Z_{k}\right) + \sum_{k=1}^{n} \dot{l}_{N}(\theta;N_k,N_{k-1}) \\
												&=& \sum_{k=1}^{n} \sum_{i\in\Omega(k\Delta)} \dot{l}_{Y}\left(\theta;Z_{i}\right) + \sum_{k=1}^{n} \dot{l}_{N}(\theta;N_k,N_{k-1}),
\eeann
\beann
\dot{l}_{Y}(\theta;z)		
&=&
\left\{
\begin{array}{ll}
\left(\frac{\partial l_{Y}(\theta;z)}{\partial K}, \frac{\partial l_{Y}(\theta;z)}{\partial\sigma}, \mathbf{0}_{1\times2}\right)^T	&  \text{if }\lambda_0\text{ is known}\\
\left(\frac{\partial l_{Y}(\theta;z)}{\partial\lambda}, \frac{\partial l_{Y}(\theta;z)}{\partial K}, \mathbf{0}_{1\times2}\right)^T	&  \text{if }\sigma_0\text{ is known,}
\end{array}
\right.
\\
\dot{l}_{N}(\theta;y,x) &=& \left(\mathbf{0}_{1\times2}, \frac{\partial l_{N}(\theta;y,x)}{\partial\alpha}, \frac{\partial l_{N}(\theta;y,x)}{\partial\mu}\right)^T,
\eeann
where the elements of $\dot{l}_{Y}(\theta;z)$ can be found in expression (\ref{2orderPartialDeriavtives}) and those of $\dot{l}_{N}(\theta;y,x)$ are given in \cite{CronieYu}.
In the integral expression, $B_n(\theta)$, of (\ref{Taylor}) we also find the symmetric matrix of second order partial derivatives
\beann
\ddot{l}_{n}(\theta)			&=& \ddot{l}_{1,n}(\theta) + \ddot{l}_{2,n}(\theta)
													= \sum_{k=1}^{n} \ddot{l}_{\Y}\left(\theta;\Z_{k}\right) + \sum_{k=1}^{n} \ddot{l}_{N}(\theta;N_k,N_{k-1}) \\
												&=& \sum_{k=1}^{n} \sum_{i\in\Omega(k\Delta)} \ddot{l}_{Y}\left(\theta;Z_{i}\right) + \sum_{k=1}^{n} \ddot{l}_{N}(\theta;N_k,N_{k-1}),
\eeann
where
\beann
\ddot{l}_{N}(\theta;y,x)
&=&
\begin{pmatrix}
\mathbf{0}_{2\times2}& \mathbf{0}_{2\times1}																					 & \mathbf{0}_{2\times1}\\
\mathbf{0}_{1\times2}& \frac{\partial^2 l_{N}(y,x;\theta)}{\partial\alpha^2}					& \frac{\partial^2 l_{N}(y,x;\theta)}{\partial\alpha\partial\mu}\\
\mathbf{0}_{1\times2}& \frac{\partial^2 l_{N}(y,x;\theta)}{\partial\mu\partial\alpha} & \frac{\partial^2 l_{N}(y,x;\theta)}{\partial\mu^2}
\end{pmatrix}
\eeann
and
\bea
\label{2orderPartialDeriavtivesMatrix}
\ddot{l}_{Y}(\theta;z)
&=&
\1\{\lambda_0\text{ known}\}
\begin{pmatrix}
\frac{\partial^2 l_{Y}(\theta;z)}{\partial K^2}								& \frac{\partial^2 l_{Y}(\theta;z)}{\partial\sigma\partial K} 			& \mathbf{0}_{1\times2}	\\
\frac{\partial^2 l_{Y}(\theta;z)}{\partial\sigma\partial K}		& \frac{\partial^2 l_{Y}(\theta;z)}{\partial\sigma^2} 							&	\mathbf{0}_{1\times2}	\\
\mathbf{0}_{2\times1}																					&	 \mathbf{0}_{2\times1}																							&	 \mathbf{0}_{2\times2}
\end{pmatrix}
\\
&&+
\1\{\sigma_0\text{ known}\}
\begin{pmatrix}
\frac{\partial^2 l_{Y}(\theta;z)}{\partial\lambda^2}					& \frac{\partial^2 l_{Y}(\theta;z)}{\partial K\partial\lambda}	& \mathbf{0}_{1\times2}	\\
\frac{\partial^2 l_{Y}(\theta;z)}{\partial K\partial\lambda}	& \frac{\partial^2 l_{Y}(\theta;z)}{\partial K^2}								 & \mathbf{0}_{1\times2}	\\
\mathbf{0}_{2\times1}																					&	 \mathbf{0}_{2\times1}																					&	 \mathbf{0}_{2\times2}
\end{pmatrix}
,\nn
\eea
and the elements of $\ddot{l}_{Y}(\theta;z)$ and $\ddot{l}_{N}(\theta;y,x)$ are given, respectively, by expression (\ref{2orderPartialDeriavtives}) and \cite{CronieYu}.
By writing the parameter vector as $\theta = (\theta_1,\theta_2,\theta_3,\theta_4)$ and consulting expression (\ref{2orderPartialDeriavtives}), then for all $j,k=1,2$, we can find bounds such that, for constants $C_{j,1},C_{j,2},C_{j,3} < \infty$ and $C_{jk,1},C_{jk,2},C_{jk,3} < \infty$ which depend on $\Theta$ (or alternatively on $S_{\rho}$),
\beann
\left|\frac{\partial l_{Y}(\theta;z)}{\partial\theta_{j}}\right|
&\leq& \sup_{\theta}\left|\frac{\partial l_{Y}(\theta;z)}{\partial\theta_{j}}\right|\\
&\leq& K_j(z) = C_{j,1}\log(z) + C_{j,2}z + C_{j,3},\\
\left|\frac{\partial^2 l_{Y}(\theta;z)}{\partial\theta_{j}\theta_{k}}\right|
&\leq& \sup_{\theta}\left|\frac{\partial^2 l_{Y}(\theta;z)}{\partial\theta_{j}\theta_{k}}\right|\\
&\leq& K_{jk}(z) = C_{jk,1}\log(z) + C_{jk,2}z + C_{jk,3},
\eeann
$\E_{\theta}[K_j(Z_1)]<\infty$ and $\E_{\theta}[K_{jk}(Z_1)]<\infty$ (recall the finiteness of $\E_{\theta}[|\log(Z_i)|]$ from the consistency proof).

Note further that differentiation under the integral sign always is permitted (i.e.\ we may interchange differential operators and integrals/expected values) since the Gamma-distribution belongs to the exponential family.
Since all partial derivatives of $l_Y(\theta;z)$ above are continuous functions, it follows that both $\dot{l}_{Y}\left(\theta;Z_{i}\right)$ and $\ddot{l}_{Y}\left(\theta;Z_{i}\right)$ are measurable (i.e.\ random variables).

\underline{Convergence of $\dot{l}_{n}(\theta_0)/\sqrt{n}$}:
We now return to expression (\ref{Taylor}) and consider the weak convergence
\bea
\label{WeakConvergence}
\frac{1}{\sqrt{n}}\dot{l}_{n}(\theta_0) &=& \frac{1}{\sqrt{n}}\dot{l}_{1,n}(\theta_0) + \frac{1}{\sqrt{n}}\dot{l}_{2,n}(\theta_0)\\
&=& \sum_{k=1}^{n}\Bigg( \frac{1}{\sqrt{n}} \underbrace{\sum_{i\in\Omega(k\Delta)}\dot{l}_{Y}(\theta_0;Z_{i})}_{=\dot{l}_{\Y}\left(\theta;\Z_{k}\right)} +  \frac{1}{\sqrt{n}}\dot{l}_{N}(\theta_0;N_k,N_{k-1})\Bigg)\nn\\
&\stackrel{d}{\longrightarrow}&
\mathbf{G}
\sim N(\mathbf{0},I(\theta_0)) =
N\left(\mathbf{0}_{4\times1},
\begin{pmatrix}
I_Y(\theta_0)					&	\mathbf{0}_{2\times2}\\
\mathbf{0}_{2\times2}	&	I_N(\theta_0)
\end{pmatrix}
\right),\nn
\eea
as $n\rightarrow\infty$, which we will prove by means of Lindeberg-Feller CLT (see e.g.\ \cite{VanDerVaart}) as opposed to the usual central limit theorem (since the component of the sum in expression (\ref{WeakConvergence}) are not identically distributed).

We start by showing that the (asymptotic) mean of $\dot{l}_{n}(\theta_0)/\sqrt{n}$ is zero.
Since we may interchange derivatives and expectations in the case of $\pi(z;\theta)$, and since $\int\pi(z;\theta_0)dz = 1$, we have that
\beann
\E_{\theta_0}\left[\dot{l}_{Y}(\theta_0;Z_{i})\right]
&=& \int \left(\frac{\partial \pi(z;\theta_0)/\partial\theta_0}{\pi(z;\theta_0)}\right) \pi(z;\theta_0) dz\\
&=& \frac{\partial}{\partial\theta_0} \int\pi(z;\theta_0)dz = \mathbf{0},
\eeann
whereby $\E_{\theta_0}[\frac{1}{\sqrt{n}}\dot{l}_{1,n}(\theta_0)] = \mathbf{0}$.
That also $\E_{\theta_0}[\frac{1}{\sqrt{n}}\dot{l}_{2,n}(\theta_0)] = \mathbf{0}$ follows since from \cite{CronieYu} we have that $\E_{\theta_0}\big[\dot{l}_{N}(\theta_0;N_i,N_{i-1})\big] = \mathbf{0}$, and thus $\E_{\theta_0}[\frac{1}{\sqrt{n}}\dot{l}_{n}(\theta_0)] = \mathbf{0}$.

Turning now to the covariance matrix $I(\theta_0)$ of expression (\ref{WeakConvergence}), from the Tower property of conditional expectations we further obtain that
\begin{align*}
&\Cov\left(\frac{1}{\sqrt{n}}\dot{l}_{1,n}(\theta_0),\frac{1}{\sqrt{n}}\dot{l}_{2,n}(\theta_0)\right) =\\
&= \frac{1}{n}\sum_{j=1}^{n}\sum_{k=1}^{n} \Cov\left(\sum_{i\in\Omega(j\Delta)} \dot{l}_{Y}(\theta_0;Z_{i}), \dot{l}_{N}(\theta_0;N_k,N_{k-1})\right)\\
&= \frac{1}{n}\sum_{j,k=1}^{n}\E_{\theta_0}\Bigg[\sum_{i\in\Omega(j\Delta)} \underbrace{\E_{\theta_0}\left[\dot{l}_{Y}(\theta_0;Z_{i})\Big|\{\Omega(m\Delta)\}_{m=1}^{n}\right]}_{=0} \dot{l}_{N}(\theta_0;N_k,N_{k-1})^T\Bigg]\\
&= \mathbf{0},
\end{align*}
which in turn implies that
\beann
\Var_{\theta_0}\left(\frac{1}{\sqrt{n}}\dot{l}_{n}(\theta_0)\right)
= \Var_{\theta_0}\left(\frac{1}{\sqrt{n}}\dot{l}_{1,n}(\theta_0)\right) + \Var_{\theta_0}\left(\frac{1}{\sqrt{n}}\dot{l}_{2,n}(\theta_0)\right).
\eeann
However, from \cite{CronieYu} in combination with \cite{Yao}) we already have that the non-zero components of $\Var_{\theta_0}\big(\dot{l}_{2,n}(\theta_0)/\sqrt{n}\big)$ converge to the Fisher information $I_N(\theta_0)$ of expression (\ref{FisherImDea}).
Hence, in order to find the Fisher information $I(\theta_0)$ of expression (\ref{WeakConvergence}), it now only remains to show that the non-zero elements of $\Var_{\theta_0}\big(\dot{l}_{1,n}(\theta_0)/\sqrt{n}\big)$ converge to $I_Y(\theta_0)$, as $n\rightarrow\infty$.
By exploiting that $\E_{\theta_0}[N(k\Delta)|N(0)=0] = \frac{\alpha_0}{\mu_0}\big(1-\e^{-\mu_0 k\Delta}\big)$ (see Lemma \ref{ImDeTransProb}) and by noticing that $\lim_{n\rightarrow\infty}\frac{1}{n}\sum_{k=1}^{n}\left(1-\e^{-\mu_0 k\Delta}\right) = 1$, we obtain
\begin{align*}
&\Var_{\theta_0}\left(\frac{1}{\sqrt{n}}\dot{l}_{1,n}(\theta_0)\right)
= \frac{1}{n}\sum_{k=1}^{n} \E_{\theta_0}\left[\sum_{i\in\Omega(k\Delta)} \dot{l}_{Y}(\theta_0;Z_{i})\dot{l}_{Y}(\theta_0;Z_{i})^T\right]\\
&= \frac{1}{n}\sum_{k=1}^{n} \E_{\theta_0}\left[\sum_{i\in\Omega(k\Delta)} \E_{\theta_0}\left[\dot{l}_{Y}(\theta_0;Z_{i})\dot{l}_{Y}(\theta_0;Z_{i})^T\bigg|\Omega(k\Delta)\right]\right]\\
&= \frac{1}{n}\sum_{k=1}^{n} \E_{\theta_0}\Bigg[\sum_{i\in\Omega(k\Delta)} \underbrace{-\E_{\theta_0}\left[\ddot{l}_{Y}(\theta_0;Z_{i})\bigg|\Omega(k\Delta)\right]}_{=I_Z(\theta_0)}\Bigg] \\
&= I_Z(\theta_0) \frac{1}{n}\sum_{k=1}^{n} \E_{\theta_0}[N_k]\\
&= I_Z(\theta_0)\frac{1}{n}\sum_{k=1}^{n}\frac{\alpha_0}{\mu_0}\left(1-\e^{-\mu_0 k\Delta}\right)
\longrightarrow I_Z(\theta_0)\frac{\alpha_0}{\mu_0}.
\end{align*}

Through expressions (\ref{2orderPartialDeriavtivesMatrix}) and (\ref{2orderPartialDeriavtives}), it can now be checked that the above covariance matrix is given by
\beann
I_Z(\theta_0)\frac{\alpha_0}{\mu_0}
&=&
\begin{pmatrix}
I_Y(\theta_0)					&	\mathbf{0}_{2\times2}\\
\mathbf{0}_{2\times2}	&	\mathbf{0}_{2\times2}
\end{pmatrix},\\
I_Y(\theta_0)
&=&
\left\{
\begin{array}{lr}
\frac{\alpha_0}{\mu_0}
\begin{pmatrix}
\frac{2\lambda_{0}}{K_{0}^2\sigma_{0}^2} & 0 \\
0 & \frac{8\lambda_{0}C(\theta_0)}{\sigma_{0}^4}
\end{pmatrix}
& \text{if }\lambda_0\text{ is known}\\
\frac{\alpha_0}{\mu_0}
\begin{pmatrix}
\frac{2 C(\theta_0)}{\lambda_{0}\sigma_{0}^2} & 0\\
0 & \frac{2\lambda_{0}}{K_{0}^2\sigma_{0}^2}
\end{pmatrix}
& \text{if }\sigma_0\text{ is known,}
\end{array}
\right.
\eeann
where $C(\theta_0) = \frac{2\lambda_0}{\sigma_0^2}\psi'\left(\frac{2\lambda_0}{\sigma_0^2}\right) - 1 > 0$ and the positive definiteness of $I_Y(\theta_0)$ follows since clearly $\y^T I_Y(\theta_0)\y = y_1^2 I_Y(\theta_0)_{11} + y_2^2 I_Y(\theta_0)_{22}>0$, for any $\y\in\R^2$, $\y\neq\mathbf{0}$.

In order to finalise the convergence of expression (\ref{WeakConvergence}) it now only remains to show that the convergence to a Gaussian law holds.
Explicitly, by means of the Lindeberg-Feller theorem (see e.g.\ \cite{VanDerVaart}), we we want to show that, as $n\rightarrow\infty$, the sum $\frac{1}{\sqrt{n}}\dot{l}_{n}(\theta_0) = \sum_{k=1}^{n}(X_k + R_k)$, where
\beann
X_k + R_k
&=& \frac{1}{\sqrt{n}}\dot{l}_{\Y}\left(\theta;\Z_{k}\right) + \frac{1}{\sqrt{n}}\dot{l}_{N}(\theta_0;N_k,N_{k-1})\\
&=& \frac{1}{\sqrt{n}}\sum_{i\in\Omega(k\Delta)}\dot{l}_{Y}(\theta_0;Z_{i}) + \frac{1}{\sqrt{n}}\dot{l}_{N}(\theta_0;N_k,N_{k-1}),
\eeann
converges in law to the Gaussian distribution in expression (\ref{WeakConvergence}).
In order to do so we need to show that the Lindeberg condition of the Lindeberg-Feller CLT is satisfied for the sequence $X_k + R_k$,
i.e.\ for every $\varepsilon>0$,
\beann
S(n) = \sum_{k=1}^{n}\E_{\theta_0}\left[\left|X_k + R_k\right|^2 \1\{\left|X_k + R_k\right|>\varepsilon\}\right] \longrightarrow 0,
\eeann
as $n\rightarrow\infty$.

We note that by the triangle inequality,
\beann
|X_k + R_k|^2 &\leq& (|X_k| + |R_k|)^2 = |X_k|^2 + |R_k|^2 + 2|X_k||R_k|,\\
\1\{\left|X_k + R_k\right|>\varepsilon\}
&\leq&
\1\{\left|X_k\right| + \left|R_k\right|>\varepsilon\}\\
&\leq& \1\{\left|X_k\right|>\varepsilon/2\} + \1\{\left|R_k\right|>\varepsilon/2\},
\eeann
and since $\Cov(X_k,R_k)=0$ and $\E_{\theta_0}\left[\1\{|R_k|>\varepsilon\}\right] = \P\left(|R_k|>\varepsilon\right)$, we obtain that
\beann
S(n)
&\leq&
\sum_{k=1}^{n}\E_{\theta_0}\left[|X_k|^2 \1\{|X_k|>\varepsilon/2\}\right]
+ \sum_{k=1}^{n}\E_{\theta_0}\left[|R_k|^2 \1\{|R_k|>\varepsilon/2\}\right]\\
&&+ \sum_{k=1}^{n}\E_{\theta_0}\left[|X_k|^2\right] \P\left(|R_k|>\varepsilon/2\right)
+ \sum_{k=1}^{n}\E_{\theta_0}\left[|R_k|^2\right] \P\left(|X_k|>\varepsilon/2\right)\\
&&+ 2\sum_{k=1}^{n}\E_{\theta_0}\left[|R_k|\right]\E_{\theta_0}\left[|X_k| \1\{|X_k|>\varepsilon/2\}\right]\\
&&+ 2\sum_{k=1}^{n}\E_{\theta_0}\left[|X_k|\right]\E_{\theta_0}\left[|R_k| \1\{|R_k|>\varepsilon/2\}\right]\\
&=& S_1(n) + S_2(n) + S_3(n) + S_4(n) + S_5(n) + S_6(n).
\eeann

We now want to show that each of the six sums above converges to zero as $n$ tends to zero.
We see that checking that the sum $S_2(n)$ tends to zero, as $n\rightarrow\infty$, is to check that the Lindeberg condition is satisfied for the equivalent convergence in the discretely sampled ID-process,
and this already holds (see the combination of \cite{CronieYu} and \cite{Yao}).

Considering the first of the sums, $S_1(n)$, instead of proving that $S_1(n)\rightarrow0$ we will prove the stronger Lyapunov condition of the Lyapunov central limit theorem (see e.g.\ \cite{Hoadley}) for the random variables $X_k$.
Denote by $\left\|X\right\|_p :=\E\left[|X|^p\right]^{\frac{1}{p}}$ the $L^p(\mathcal{X},\mathcal{F},\P)$-norm of the random variable $X$ and recall the bounds $K_j(Z_1)$, $j=1,2$, of the elements of $\dot{l}_{Y}(\theta_0;Z_{i})$.
Given $p>2$ and any $\y\in\R^4$, by the conditional version of Minkowski's inequality, we have that
\begin{align*}
&\E_{\theta_0}\left[\left|\y^T X_k\right|^{p}\Big|\Omega(k\Delta)\right]
= \left\|\left(\frac{1}{\sqrt{n}}\sum_{i\in\Omega(k\Delta)} \y^T \dot{l}_{Y}(\theta_0;Z_{i})\Bigg|\Omega(k\Delta)\right)\right\|_p^p\\
&= \frac{1}{n^{p/2}}\left\|\left(\sum_{i\in\Omega(k\Delta)} \sum_{j=1}^{2} y_j\dot{l}_{Y}(\theta_0;Z_{i})_j\Bigg|\Omega(k\Delta)\right)\right\|_p^p\\
&\leq \frac{1}{n^{p/2}}\sum_{i\in\Omega(k\Delta)} \sum_{j=1}^{2}|y_j|^p \left\|\left(\dot{l}_{Y}(\theta_0;Z_{i})_j\Big|\Omega(k\Delta)\right)\right\|_p^p\\
&= \frac{1}{n^{p/2}}\sum_{i\in\Omega(k\Delta)} \sum_{j=1}^{2}|y_j|^p \E_{\theta_0}\left[\left|\dot{l}_{Y}(\theta_0;Z_{i})_j\right|^p\bigg|\Omega(k\Delta)\right]\\
&\leq \frac{1}{n^{p/2}}\sum_{i\in\Omega(k\Delta)} \sum_{j=1}^{2}|y_j|^p \E_{\theta_0}\left[K_j(Z_{1})^p\right],
\end{align*}
which in turn implies that
\beann
\E_{\theta_0}\left[\left|\y^T X_k\right|^{p}\right]
&=& \E_{\theta_0}\left[\E_{\theta_0}\left[\left|\y^T X_k\right|^{p}\Big|\Omega(k\Delta)\right]\right] \\
&\leq& \frac{1}{n^{p/2}} \E_{\theta_0}[N_k] \sum_{j=1}^{2}|y_j|^p \E_{\theta_0}\left[K_j(Z_{1})^p\right]\\
&\leq& \frac{1}{n^{p/2}} \frac{\alpha_0}{\mu_0}(1-\e^{-\mu_0 k\Delta}) \sum_{j=1}^{2}|y_j|^p \E_{\theta_0}\left[K_j(Z_{1})^p\right].
\eeann
We will now deal with this expression when $p=4$.
In the case of $p=4$ the finiteness of $\E_{\theta_0}\left[K_j(Z_{i})^p\right]$ holds for any compact parameter space $\Theta$ (alternatively $S_{\rho}$) since $\E_{\theta_0}[\log(Z_1)^4]$ and $\E_{\theta_0}[Z_1^4]$ are finite (their expressions can be found in Section \ref{SectionAppendixDerivatives} in the Appendix).
By further also noticing that $\sum_{j=1}^{2}|y_j|^4 < \infty$, $\frac{\alpha_0}{\mu_0} < \infty$ and by recalling that $\frac{1}{n}\sum_{k=1}^{n}\left(1-\e^{-\mu_0 k\Delta}\right)\rightarrow1$,
when choosing $p=4$ and letting $n\rightarrow\infty$, we finally obtain that the right hand side of the above expression tends to zero, whereby $\E_{\theta_0}\left[\left|\y^T X_k\right|^{p}\right]\rightarrow0$.
Hence, the Lyapunov condition $\sum_{k=1}^{n}\E_{\theta_0}\left[\left|\y^T X_k\right|^{4}\right] \rightarrow 0$ is satisfied,
and hereby the Lindeberg condition $S_1(n)\rightarrow0$ follows.

The next convergence proved is $\lim_{n\rightarrow\infty}S_3(n) = 0$.
From the above derivations, we now additionally see that
\beann
\lim_{n\rightarrow\infty} \sum_{k=1}^{n}\E_{\theta_0}\left[|X_k|^{2}\right]
&\leq&
\frac{\alpha_0}{\mu_0} \sum_{j=1}^{2} \E_{\theta_0}\left[K_j(Z_{1})^2\right] \lim_{n\rightarrow\infty}\frac{1}{n} \sum_{k=1}^{n} (1-\e^{-\mu_0 k\Delta}) \\
&=&
\frac{\alpha_0}{\mu_0} \sum_{j=1}^{2} \E_{\theta_0}\left[K_j(Z_{1})^2\right] < \infty.
\eeann
Given $\rho(k) = \frac{\alpha_0}{\mu_0}(1-\e^{-\mu_0 k\Delta})$, recall from Lemma \ref{ImDeTransProb} that $\E[N(h+t)|N(h)=i]=i\e^{-\mu t} + \rho$ and $\E[N^2(h+t)|N(h)=i] = i(i-1)\e^{-2\mu t} + (1+2\rho)i\e^{-\mu t} + \rho^2 + \rho$.
By exploiting Markov's inequality and considering bounds given in \cite{CronieYu}, for each $k=1,\ldots,n$, we have that
\begin{align*}
&\P\left(|R_k|>\varepsilon/2\right)
< \frac{1}{\varepsilon/2} \E_{\theta_0}[|R_k|]\\
&= \frac{1}{\sqrt{n}\varepsilon/2} \E_{\theta_0}[|\dot{l}_{N}(\theta_0;N_k,N_{k-1})|]\\
&< \frac{1}{\sqrt{n}\varepsilon/2} \E_{\theta_0}\left[\frac{N_k}{\alpha_0} + \Delta + \frac{\alpha_0\Delta^2 + (3N_k + N_{k-1})\Delta}{1-\e^{-\mu_0\Delta}}\right]\\
&= \frac{2\Delta}{\sqrt{n}\varepsilon} \left(\frac{\E_{\theta_0}[N_k]}{\Delta\alpha_0} + 1 + \frac{\alpha_0\Delta + (3\E_{\theta_0}[N_k] + \E_{\theta_0}[N_{k-1}])}{1-\e^{-\mu_0\Delta}}\right)\\
&= \frac{2\Delta}{\sqrt{n}\varepsilon}
\left(\frac{1-\e^{-\mu_0 k\Delta}}{\Delta\mu_0} + 1 + \alpha_0\frac{\Delta + \frac{1}{\mu_0}(3(1-\e^{-\mu_0 k\Delta}) + (1-\e^{-\mu_0 (k-1)\Delta}))}{1-\e^{-\mu_0\Delta}}\right)\\
&< \frac{2\Delta}{\sqrt{n}\varepsilon}
\left(\frac{1}{\Delta\mu_0} + 1 + \alpha_0\frac{\Delta + \frac{4}{\mu_0}}{1-\e^{-\mu_0\Delta}}\right)
= \frac{2}{\sqrt{n}\varepsilon}C_{1,R}
\longrightarrow 0,
\end{align*}
as $n\rightarrow\infty$, whence $\sup_{1\leq k \leq n}\P\left(|R_k|>\varepsilon\right) \rightarrow 0$.
It now readily follows that
\beann
\lim_{n\rightarrow\infty}S_3(n) \leq \lim_{n\rightarrow\infty}\sup_{1\leq k \leq n}\P\left(|R_k|>\varepsilon\right) \sum_{k=1}^{n}\E_{\theta_0}\left[|X_k|^{2}\right] = 0.
\eeann
We note further that from the above arguments it also follows that $\E_{\theta_0}[|X_k|] < \infty$ and $\E_{\theta_0}[|R_k|] < \infty$.

We now turn to
\beann
S_4(n) = \sum_{k=1}^{n}\E_{\theta_0}\left[|R_k|^2\right] \P\left(|X_k|>\varepsilon/2\right).
\eeann
From Lemma \ref{ImDeTransProb} we find that
\beann
\E_{\theta_0}[N_k N_{k-1}] &=& \E_{\theta_0}[\E_{\theta_0}[N_k|N_{k-1}]N_{k-1}]\\
&=& \E_{\theta_0}[(N_{k-1}\e^{-\mu_0 k\Delta} + \rho(k))N_{k-1}]\\
&=& \rho(k-1)(\rho(k-1)+1)\e^{-\mu_0 k\Delta} + \rho(k)\rho(k-1)\\
&<& 2\frac{\alpha_0}{\mu_0}\left(\frac{\alpha_0}{\mu_0} + 1\right),
\eeann
and since $\rho(k)<\alpha_0/\mu_0$ for any $k$, we note that according to \cite{CronieYu} we have that
\begin{align*}
&\E_{\theta_0}\left[|R_k|^2\right]
=
\frac{1}{n}\E_{\theta_0}\left[\left|\dot{l}_{N}(\theta_0;N_k,N_{k-1})\right|^2 \right]\\
&< \frac{1}{n}\E_{\theta_0}\left[\left(\frac{N_k}{\alpha_0} + \Delta\right)^2 + \left(\frac{\alpha_0\Delta^2 + (3N_k + N_{k-1})\Delta}{1-\e^{-\mu_0\Delta}}\right)^2\right]\\
&= \frac{1}{n}\Bigg(\frac{1}{\alpha_0^2}\E_{\theta_0}\left[N_k^2\right] + \Delta^2 + \frac{\alpha_0^2\Delta^4}{(1-\e^{-\mu_0\Delta})^2}
 + \frac{2\alpha_0\Delta^3(3\E_{\theta_0}[N_k] + \E_{\theta_0}[N_{k-1}])}{(1-\e^{-\mu_0\Delta})^2}\\
&+ \frac{\Delta^2(9\E_{\theta_0}[N_k^2] + 6\E_{\theta_0}[N_{k}N_{k-1}] + \E_{\theta_0}[N_{k-1}^2])}{(1-\e^{-\mu_0\Delta})^2}
\Bigg)\\
&< \frac{1}{n}\Bigg(\frac{1}{\alpha_0^2}\frac{\alpha_0}{\mu_0}\left(\frac{\alpha_0}{\mu_0} + 1\right) + \Delta^2 + \frac{\alpha_0^2\Delta^4}{(1-\e^{-\mu_0\Delta})^2}\\
& + \frac{8\alpha_0\Delta^3\frac{\alpha_0}{\mu_0} + \Delta^2 22\frac{\alpha_0}{\mu_0}\left(\frac{\alpha_0}{\mu_0} + 1\right)}{(1-\e^{-\mu_0\Delta})^2}
\Bigg)
= \frac{1}{n}\underbrace{C_{2,R}}_{<\infty}.
\end{align*}
By Markov's inequality we have that
\beann
\P\left(|X_k|>\varepsilon/2\right)
&<& \frac{1}{\varepsilon/2} \E_{\theta_0}[|X_k|]\\
&\leq& \frac{1}{\sqrt{n}\varepsilon/2} \frac{\alpha_0}{\mu_0}(1-\e^{-\mu_0 k\Delta}) \sum_{j=1}^{2} \E_{\theta_0}\left[K_j(Z_{1})^p\right],
\eeann
whereby, as $n\rightarrow\infty$, it follows that
\beann
S_4(n) < \frac{1}{\sqrt{n}\varepsilon/2}\underbrace{\frac{1}{n} \sum_{k=1}^{n}(1-\e^{-\mu_0 k\Delta})}_{\rightarrow 1}
\underbrace{\sum_{j=1}^{2} \E_{\theta_0}\left[K_j(Z_{1})^p\right]C_{2,R}\frac{\alpha_0}{\mu_0}}_{<\infty} \longrightarrow 0.
\eeann

Dealing finally with the convergence to zero of the last two sums,
by recalling the bound $\E_{\theta_0}[|R_k|]<\frac{1}{\sqrt{n}}C_{1,R}$, we now see that
\beann
S_5(n)
&=& 2\sum_{k=1}^{n}\E_{\theta_0}[|R_k|] \E_{\theta_0}\left[|X_k| \1\{|X_k|>\varepsilon/2\}\right]\\
&<&
\frac{2}{\sqrt{n}}C_{1,r}
\sum_{k=1}^{n} \E_{\theta_0}\left[\frac{|X_k|^2}{\varepsilon/2} \1\{|X_k|>\varepsilon/2\}\right]\\
&=& \frac{4}{\sqrt{n}\varepsilon}C_{1,r} S_1(n)
\longrightarrow 0,
\eeann
as $n\rightarrow\infty$.
Similarly, as $n\rightarrow\infty$, we also obtain
\beann
S_6(n)
&=& 2\sum_{k=1}^{n} \E_{\theta_0}[|X_k|]\E_{\theta_0}\left[|R_k| \1\{|R_k|>\varepsilon/2\}\right]\\
&\leq&
\frac{4}{\sqrt{n}\varepsilon} \frac{\alpha_0}{\mu_0} \sum_{j=1}^{2} \E_{\theta_0}\left[K_j(Z_{1})\right]\\
&&\times \sum_{k=1}^{n} (1-\e^{-\mu_0 k\Delta}) \E_{\theta_0}\left[|R_k|^2 \1\{|R_k|>\varepsilon/2\}\right]\\
&<&
\frac{1}{\sqrt{n}} \underbrace{\frac{4\alpha_0}{\varepsilon\mu_0} \sum_{j=1}^{2} \E_{\theta_0}\left[K_j(Z_{1})\right]}_{<\infty}
S_2(n) \longrightarrow 0.
\eeann

Hence, as $\varepsilon>0$ was arbitrary, we conclude that $\lim_{n\rightarrow\infty}S(n) = 0$, and hereby the convergence of expression (\ref{WeakConvergence}) follows.\\

\underline{The a.s.\ convergence of $-B_{n}(\hat{\theta}_{n})$ to $I(\theta_0)$}:
We first note that by writing
\beann
I(\theta_0) = \tilde{I}_Y(\theta_0) + \tilde{I}_N(\theta_0)
=
\begin{pmatrix}
I_Y(\theta_0)							& \mathbf{0}_{2\times2}\\
 \mathbf{0}_{2\times2}		& \mathbf{0}_{2\times2}
\end{pmatrix}
+
\begin{pmatrix}
\mathbf{0}_{2\times2}	& \mathbf{0}_{2\times2}\\
\mathbf{0}_{2\times2}	& I_N(\theta_0)
\end{pmatrix}
\eeann
we have that
\begin{align}
\left|-B_{n}(\hat{\theta}_{n}) - I(\theta_0)\right|
&=			\left|\frac{1}{n} \int_{0}^{1} \ddot{l}_{n}(\theta_0 + x(\hat{\theta}_{n}-\theta_0))dx + I(\theta_0)\right|\nn\\
&\leq	\left|\int_{0}^{1} \frac{1}{n} \sum_{k=1}^{n} \ddot{l}_{\Y}(\theta_0 + x(\hat{\theta}_{n}-\theta_0);\Z_{k})dx + \tilde{I}_Y(\theta_0)\right|\nn\\
&+ 		\left|\int_{0}^{1} \frac{1}{n} \sum_{k=1}^{n} \ddot{l}_{N}(\theta_0 + x(\hat{\theta}_{n}-\theta_0);N_k,N_{k-1})dx + \tilde{I}_N(\theta_0)\right|\nn\\
&=			D_Y(n) + D_N(n).\nn
\end{align}

Denote by $l_n(\alpha_0,\mu_0)$ the ID-process likelihood and by $B_{n}^N(\hat{\alpha}_{n},\hat{\mu}_{n})$ the remainder term used in the Taylor expansion of the corresponding proof of the asymptotic normality in \cite{CronieYu};
\beann
\sqrt{n}\left((\hat{\alpha}_{n}, \hat{\mu}_{n})-(\alpha_{0}, \mu_{0})\right)^T = \left(-B_{n}^N(\hat{\alpha}_{n},\hat{\mu}_{n})\right)^{-1}\frac{1}{\sqrt{n}}l_n(\alpha_0,\mu_0).
\eeann
That $D_N(n)\stackrel{a.s.}{\longrightarrow}0$, or equivalently that $-B_{n}^N(\hat{\alpha}_{n},\hat{\mu}_{n})\stackrel{a.s.}{\longrightarrow}I_N(\theta_0)$, as $n\rightarrow\infty$, then follows from  \cite{CronieYu} (in combination with \cite{Yao} (Theorem 2)).

Hence, it now only remains to show that $D_Y(n)\stackrel{a.s.}{\longrightarrow}0$.
From the dominated convergence theorem (recall the continuity and the bounds of $\ddot{l}_{Y}(\theta_0;z)$) we obtain the $\theta$--continuity of $\E_{\theta_0}\big[\ddot{l}_{Y}(\theta;Z_{i})\big]$, whereby the continuity of $\gamma(\theta):=\frac{\alpha_0}{\mu_0}\E_{\theta_0}\big[\ddot{l}_{Y}(\theta;Z_{i})\big]$ follows.
Hereby, for every $\varepsilon>0$, there exists a $\rho>0$ such that for $\theta \in S_{\rho} = \{\theta\in\Theta:d(\theta,\theta_0)\leq\rho\}$ we have
\beann
\left|\gamma(\theta) - \gamma(\theta_0)\right|
=
\left|\gamma(\theta) + \tilde{I}_Y(\theta_0)\right| < \varepsilon/2.
\eeann

Furthermore, from \cite{CronieYu} we have the following strong law of large numbers for $N(t)$:
As $n\rightarrow\infty$, for any $\pi_N(\cdot;\theta_0)$-integrable function $\vartheta:\N\rightarrow\R$ we obtain
\beann
\frac{1}{n}\sum_{k=1}^{n}\vartheta(N(k\Delta)) \stackrel{a.s.}{\longrightarrow} \sum_{x\in\N}\vartheta(x)\pi_{N}(x;\theta_0).
\eeann
By combining this with Lemma \ref{ImDeTransProb} we further obtain that
$\frac{A(n)}{n} = \frac{1}{n}\sum_{k=1}^{n}N_k \stackrel{a.s.}{\longrightarrow}\sum_{x\in\N}x\pi_{N}(x;\theta_0) = \frac{\alpha_0}{\mu_0}$.
Through Slutsky's lemma (see e.g.\ \cite{Ferguson}), when combining this convergence with Lemma \ref{USLLN} (recall the continuity  and the bounds related to $\ddot{l}_{Y}(\theta;Z_{i})$), we get that
\beann
\sup_{\theta\in S_{\rho}} \left|\frac{1}{n}\sum_{i=1}^{A(n)} \ddot{l}_{Y}(\theta;Z_{i}) - \frac{A(n)}{n}\E_{\theta_0}\left[\ddot{l}_{Y}(\theta;Z_{i})\right]\right| &=&\\
= \sup_{\theta\in S_{\rho}} \left|\frac{1}{A(n)}\sum_{i=1}^{A(n)} \ddot{l}_{Y}(\theta;Z_{i}) - \E_{\theta_0}\left[\ddot{l}_{Y}(\theta;Z_{i})\right]\right|\frac{A(n)}{n}
&\stackrel{a.s.}{\longrightarrow}& 0.
\eeann
But since
$\frac{A(n)}{n}\E_{\theta_0}\big[\ddot{l}_{Y}(\theta;Z_{i})\big] \stackrel{a.s.}{\longrightarrow} \frac{\alpha_0}{\mu_0}\E_{\theta_0}\big[\ddot{l}_{Y}(\theta;Z_{i})\big]$,
as $n\rightarrow\infty$, we have that
\beann
\sup_{\theta\in S_{\rho}} \left|\frac{1}{n}\sum_{i=1}^{A(n)} \ddot{l}_{Y}(\theta;Z_{i}) - \frac{\alpha_0}{\mu_0}\E_{\theta_0}\big[\ddot{l}_{Y}(\theta;Z_{i})\big]\right| \stackrel{a.s.}{\longrightarrow} 0.
\eeann
Hence, there a.s.\ is some integer $n^*$ such that $n>n^*$ implies that
\beann
\sup_{\theta\in S_{\rho}} \left|\frac{1}{n}\sum_{i=1}^{A(n)} \ddot{l}_{Y}(\theta;Z_{i}) - \gamma(\theta)\right| < \varepsilon/2.
\eeann

Now, by choosing $n^*$ large enough to have $\hat{\theta}_{n}\in S_{\rho}$, when $n>n^*$ we have that
\beann
B_Y(n)
&\leq& \int_{0}^{1} \left|\frac{1}{n} \sum_{i=1}^{A(n)} \ddot{l}_{Y}(\theta_0 + x(\hat{\theta}_{n}-\theta_0);Z_{i}) + \tilde{I}_Y(\theta_0)\right|dx\\
&\leq& \int_{0}^{1} \sup_{\theta\in S_{\rho}}
\left\{
\left|\frac{1}{n}\sum_{i=1}^{A(n)} \ddot{l}_{Y}(\theta;Z_{i}) - \gamma(\theta)\right| + \left|\gamma(\theta) + \tilde{I}_Y(\theta_0)\right|\right\}dx \leq \varepsilon.
\eeann

We now finally find the inverse of the Fisher information, $I(\theta_0)^{-1}$, which is given by the covariance matrix of the asymptotic Gaussian distribution of $\sqrt{n}(\hat{\theta}_n - \theta_0)$ in the statement of the theorem.

\end{proof}

\subsection{Log-likelihood derivatives}\label{SectionAppendixDerivatives}

When $Z\sim\Gamma(2\lambda_0/\sigma_0^2,\sigma_0^2 K_0/2\lambda_0)$ we have that
\[
\begin{array}{lll}
\E_{\theta_0}[Z] &=& K_0,\\
\E_{\theta_0}[Z^4] &=& K_0^4(\lambda_0 + \sigma_0^2)(2\lambda_0 + \sigma_0^2)(2\lambda_0 + 3\sigma_0^2)/4\lambda_0^3,\\
\E_{\theta_0}[\log Z] &=& \log(K_0\sigma_0^2/2\lambda_0) + \psi(2\lambda_0/\sigma_0^2),\\
\E_{\theta_0}[\log(Z)^2] &=&
\log\left(\frac{2\lambda_0}{K_0\sigma_0^2}\right)^2 - 2\log\left(\frac{2\lambda_0}{K_0\sigma_0^2}\right)\psi\left(\frac{2\lambda_0}{\sigma_0^2}\right)
+ \psi\left(\frac{2\lambda_0}{\sigma_0^2}\right)^2\\
&&+ \psi'\left(\frac{2\lambda_0}{\sigma_0^2}\right),
\end{array}
\]
where $\psi(x)=\Gamma'(x)/\Gamma(x)$, $\Gamma(x)$ is the gamma function, and\\
\begin{align}
&\E_{\theta_0}[\log(Z)^4]
= \log(2)^4
+ 4\log(2)^3\log\left(\frac{\lambda_{0}}{K_{0}\sigma_{0}^2}\right)
+ 6\log(2)^2\log\left(\frac{\lambda_{0}}{K_{0}\sigma_{0}^2}\right)^2\nn\\
&+ \log(16)\log\left(\frac{\lambda_{0}}{K_{0}\sigma_{0}^2}\right)^3
+ \log\left(\frac{\lambda_{0}}{K_{0}\sigma_{0}^2}\right)^4
- 4\log\left(\frac{2\lambda_{0}}{K_{0}\sigma_{0}^2}\right)\psi\left(\frac{2\lambda_{0}}{\sigma_{0}^2}\right)^3\nn\\
&+ \psi\left(\frac{2\lambda_{0}}{\sigma_{0}^2}\right)^4
+ 6 \log\left(\frac{2\lambda_{0}}{K_{0}\sigma_{0}^2}\right)^2\psi'\left(\frac{2\lambda_{0}}{\sigma_{0}^2}\right)
+ 3\psi'\left(\frac{2\lambda_{0}}{\sigma_{0}^2}\right)^2 \nn\\
&+ 6\psi\left(\frac{2\lambda_{0}}{\sigma_{0}^2}\right)^2\left(\log\left(\frac{2\lambda_{0}}{K_{0}\sigma_{0}^2}\right)^2
+ \psi'\left(\frac{2\lambda_{0}}{\sigma_{0}^2}\right)\right)\nn\\
&- 4\psi\left(\frac{2\lambda_{0}}{\sigma_{0}^2}\right)
\Bigg[\log\left(\frac{2\lambda_{0}}{K_{0}\sigma_{0}^2}\right)^3\nn\\
&+ \left(\log(8) + 3\log\left(\frac{\lambda_{0}}{K_{0}\sigma_{0}^2}\right)\right)\psi'\left(\frac{2\lambda_{0}}{\sigma_{0}^2}\right)
- \psi''\left(\frac{2\lambda_{0}}{\sigma_{0}^2}\right)\Bigg]\nn\\
&- 4\log(2)\psi''\left(\frac{2\lambda_{0}}{\sigma_{0}^2}\right)
- 4\log\left(\frac{\lambda_{0}}{K_{0}\sigma_{0}^2}\right)\psi''\left(\frac{2\lambda_{0}}{\sigma_{0}^2}\right)
+ \psi'''\left(\frac{2\lambda_{0}}{\sigma_{0}^2}\right).\nn
\end{align}

Note that these expressions readily can be obtained by using the software \emph{Mathematica}.

Furthermore, by writing $l_{Y}(\theta;z) = \log\pi(z;\lambda,K,\sigma)$ for simplicity, the partial derivatives of $l_{Y}(\theta;z)$ w.r.t. $\lambda$, $K$ and $\sigma$ are given by
\bea
\label{2orderPartialDeriavtives}
\frac{\partial l_{Y}(\theta;z)}{\partial\lambda}
&=& \frac{\log(z) - \frac{z}{K} + 1 + \log(2) - \log\left(\frac{K\sigma^2}{\lambda}\right) - \psi\left(\frac{2\lambda}{\sigma^2}\right)}{\sigma^2/2}\\
\frac{\partial l_{Y}(\theta;z)}{\partial K} 											&=& \frac{2\lambda(z-K)}{\sigma^2K^2}\nn\\
\frac{\partial l_{Y}(\theta;z)}{\partial\sigma}
&=& \frac{\frac{z}{K} - \log(z) - \log(2) - 1 + \log\left(\frac{K\sigma^2}{\lambda}\right) + \psi\left(\frac{2\lambda}{\sigma^2}\right)}{\sigma^3/4\lambda}\nn\\
\frac{\partial^2 l_{Y}(\theta;z)}{\partial\lambda^2} 							&=& \frac{2}{\lambda\sigma^4}\left(\sigma^2 - 2\lambda \psi'\left(\frac{2\lambda}{\sigma^2}\right)\right), \qquad
\frac{\partial^2 l_{Y}(\theta;z)}{\partial\lambda \partial K} 		= \frac{2(z - K)}{\sigma^{2}K^{2}}\nn\\
\frac{\partial^2 l_{Y}(\theta;z)}{\partial\lambda \partial\sigma} &=& \frac{\frac{z}{K} - \log(z) - 2 - \log(2) + \log\left(\frac{\sigma^{2} K}{\lambda}\right) +
																							 \psi\left(\frac{2\lambda}{\sigma^2}\right) + \frac{2\lambda}{\sigma^2}\psi'\left(\frac{2\lambda}{\sigma^2}\right)}{\sigma^{3}/4}\nn\\
\frac{\partial^2 l_{Y}(\theta;z)}{\partial K\partial\sigma} 			&=& \frac{4\lambda(K - z)}{\sigma^3 K^2},\qquad
\frac{\partial^2 l_{Y}(\theta;z)}{\partial K^2} 									= \frac{2\lambda(K - 2z)}{\sigma^{2} K^3}\nn\\
\frac{\partial^2 l_{Y}(\theta;z)}{\partial\sigma^2} 							&=& \frac{-\frac{z}{K} + \log(z) + \frac{5 + \log(8)}{3} - \log\left(\frac{\sigma^2 K}{\lambda}\right) -
																				 \psi\left(\frac{2\lambda}{\sigma^2}\right) - \frac{4\lambda}{3\sigma^2}\psi'\left(\frac{2\lambda}{\sigma^2}\right)}{\sigma^4/12\lambda}.\nn
\eea

\bibliographystyle{plain}
\bibliography{references}

\end{document}